\documentclass[11pt]{amsart}
\usepackage{amscd}
\usepackage{pstricks}
\usepackage[all]{xy}

\newtheorem{theo}{Theorem}[section]
\newtheorem{lemma}[theo]{Lemma}
\newtheorem{defi}[theo]{Definition}
\newtheorem{prop}[theo]{Proposition}

\newtheorem{cor}[theo]{Corollary}
\newtheorem{remark}[theo]{Remark}

\numberwithin{equation}{section}

\def\R{\mathbb{R}}
\def\C{\mathbb{C}}
\def\Z{\mathbb{Z}}
\def\W{\mathcal{W}}
\def\kk{\mathbf{k}}
\def\co{\colon\thinspace}

\def\db#1{D^b({#1})}
\def\coh{\operatorname{coh}}
\def\lto{\longrightarrow}

\def\bdot{{\scriptscriptstyle\bullet}}

\def\PP{{\mathbb P}}

\def\pre-tr{\operatorname{pre-tr}}

\def\Hom{\operatorname{Hom}}
\def\End{\operatorname{End}}

\def\Ob{\operatorname{Ob}}
\def\perf#1{{\mathfrak P}{\mathfrak e}{\mathfrak r}{\mathfrak f}({#1})}

\newcommand{\bbA}{{\mathbb A}}
\newcommand{\bbC}{{\mathbb C}}

\newcommand{\bbZ}{{\mathbb Z}}

\newcommand{\bbP}{{\mathbb P}}

\newcommand{\cO}{{\mathcal O}}

\newcommand{\cN}{{\mathcal N}}

\newcommand{\cA}{{\mathcal A}}

\newcommand{\cC}{{\mathcal C}}

\newcommand{\cT}{{\mathcal T}}

\newcommand{\coker}{\operatorname{Coker}}

\newcommand{\Ext}{\operatorname{Ext}}

\newcommand{\id}{\operatorname{id}}

\newcommand{\Gr}{\operatorname{Gr}}

\title{Homological mirror symmetry for punctured spheres}

\author[M.~Abouzaid]{Mohammed Abouzaid}
\address{Department of Mathematics, Columbia University, 2990 Broadway, 
New York, NY 10027, USA}
\email{abouzaid@math.columbia.edu}

\author[D.~Auroux]{Denis Auroux}
\address{Department of Mathematics, UC Berkeley, Berkeley CA 94720-3840,
USA}
\email{auroux@math.berkeley.edu}

\author[A.~I.~Efimov]{Alexander I. Efimov}
\address{
Algebraic Geometry Section, Steklov Mathematical Institute, Russian Academy
of Sciences, 8 Gubkin str., Moscow 119991, Russia}
\email{efimov13@yandex.ru}

\author[L.~Katzarkov]{Ludmil Katzarkov}
\address{
Department of Mathematics, Universit\"at Wien, Garnisongasse 3, Vienna
A-1090, Austria
}
\email{ludmil.katzarkov@univie.ac.at}

\author[D.~Orlov]{Dmitri Orlov}
\address{
Algebraic Geometry Section, Steklov Mathematical Institute, Russian Academy
of Sciences, 8 Gubkin str., Moscow 119991, Russia}
\email{orlov@mi.ras.ru}

\begin{document}

\begin{abstract}
We prove that the wrapped Fukaya category of a punctured sphere
($S^{2}$ with an arbitrary number of points removed) is equivalent to
the triangulated category of singularities of a mirror Landau-Ginzburg model, proving one side of the homological mirror symmetry conjecture in this case.  By investigating fractional gradings on these categories, we conclude that cyclic covers on the symplectic side are mirror to orbifold quotients of the Landau-Ginzburg model.
\end{abstract}

\maketitle

\section{Introduction}

\subsection{Background}
In its original formulation, Kontsevich's celebrated homological mirror symmetry
conjecture \cite{KoICM} concerns mirror pairs of Calabi-Yau varieties,
for which it predicts an equivalence between the derived category of
coherent sheaves of one variety and the derived Fukaya category of the
other. This conjecture has been studied extensively, and while evidence has
been gathered in a number of examples including abelian varieties
\cite{Fuk,KS,KO}, it has so far only been proved for elliptic curves \cite{PZ}, the quartic K3 surface
\cite{SeK3}, and their products \cite{ASmith}.

Kontsevich was also the first to suggest that homological mirror
symmetry can be extended to a much more general setting \cite{KoENS}, by
considering {\em Landau-Ginzburg models}. Mathematically, a Landau-Ginzburg model
is a pair $(X,W)$ consisting of a variety $X$ and a
holomorphic function $W:X\to\C$ called {\it superpotential}. As far as
homological mirror symmetry is concerned, the symplectic geometry of a
Landau-Ginzburg model is determined by its Fukaya category, studied
extensively by Seidel (see in particular \cite{seidel-book}), while the
B-model is determined by the triangulated category of singularities of
the superpotential \cite{Orlov}.

After the seminal works of Batyrev, Givental, Hori, Vafa, and many others, there are many known examples of Landau-Ginzburg mirrors to Fano varieties, especially in the toric case \cite{Cho,FOOO,CL} where the examples can be understood using T-duality, generalising the ideas of Strominger, Yau, and Zaslow \cite{SYZ} beyond the case of Calabi-Yau manifolds.     One direction of the mirror symmetry conjecture, in which the $B$-model consists of coherent sheaves on a Fano variety, has been established for toric Fano varieties in \cite{bondal-ruan, AKO1, Ueda, AbToric,FLTZ}, as well as for  for del Pezzo surfaces \cite{AKO2}.  A proof in the other direction, in which the $B$-model is the category of matrix
factorizations of the superpotential,  has also been announced \cite{AFOOO}.

While Kontsevich's suggestion was originally studied for Fano manifolds, a
more recent (and perhaps unexpected) development  first
proposed by the fourth author is that mirror symmetry also extends to varieties of
general type, many of which also admit mirror Landau-Ginzburg models
\cite{KKOY,AAK,GKR}. The first instance of homological mirror symmetry
in this setting was established for the genus 2 curve by
Seidel~\cite{Segenus2}. Namely, Seidel has shown that
the derived Fukaya category of a smooth genus 2 curve is
equivalent to the triangulated category of singularities of a certain
3-dimensional Landau-Ginzburg model (one notable feature of mirrors of
varieties of general type is that they tend to be higher-dimensional).
Seidel's argument was subsequently extended to higher genus curves
\cite{efimov}, to pairs of pants and their higher-dimensional analogues
\cite{sheridan}, and to Calabi-Yau hypersurfaces in projective space \cite{sheridan-2}.

Unfortunately, the ordinary Fukaya category consisting of closed Lagrangians
is insufficient in order to  fully state the Homological mirror conjecture when the $B$-side is
a Landau-Ginzburg model which fails to be proper or a variety which fails
to be smooth.  The structure sheaf of a non-proper component of the critical fiber of
a Landau-Ginzburg model, or that of a singular point in the absence of any superpotential,
generally have endomorphism algebras which are not of finite cohomological dimension, and
hence cannot have mirrors in the ordinary Fukaya category, which is cohomologically finite.
As all smooth affine varieties of the same dimension have isomorphic derived categories
of coherent sheaves with compact support, one is led to seek a category
of Lagrangians which would contain objects that are mirror to more
general sheaves or matrix factorizations.

It is precisely to fill this role that the {\it wrapped Fukaya category}
was constructed \cite{AS}. This Fukaya category, whose objects
also include non-compact Lagrangian submanifolds, more accurately reflects
the symplectic geometry of open symplectic manifolds, and  by
recent work \cite{generate,ganatra}, is known in some generality
 to be homologically smooth in the sense of
Kontsevich \cite{KS-notes} (homological smoothness also holds for categories of matrix factorizations \cite{lunts,preygel,LP}).

In this paper, we give the first non-trivial verification that these categories
are indeed relevant to Homological mirror symmetry:  the non-compact Lagrangians
we shall study will correspond to structure sheaves of irreducible components
of a quasi-projective variety, considered as objects of its category of singularities.
In particular, we provide the first computation of wrapped Fukaya categories beyond the
case of cotangent bundles, studied in \cite{A-cotangent} using string
topology. Since the writing of this paper, Bocklandt found a connection to
non-commutative algebras coming from dimer models which allows an
extension of our results to general punctured surfaces \cite{bocklandt}.

As a final remark, we note that these categories should be of interest even
when considering mirrors of compact symplectic manifolds.  Indeed, since Seidel's
ICM  address \cite{seidel-ICM}, the standard approach to proving Homological mirror
symmetry in this case is to first prove it for the complement of a divisor, then
solve a deformation problem.  As we have just explained, a proper formulation
of Homological mirror symmetry for the complement involves the wrapped Fukaya category.
More speculatively \cite{seidel-speculation}, one expects that the study of the
wrapped Fukaya category will be amenable to sheaf-theoretic techniques.  The starting
point of such a program is  the availability of
natural restriction functors (to open subdomains)~\cite{AS}, which are expected to
be mirror to restriction functors from the category of sheaves of a reducible variety
to the category of sheaves on each component. This suggests that it might be possible to study homological mirror symmetry by a combination of sheaf-theoretic techniques and deformation theory, reducing the problem to elementary building blocks
such as pairs of pants. While this remains a distant perspective, it
very much motivates the present study.

\subsection{Main results}

In this paper, we study homological mirror symmetry for an open genus~0
curve $C,$ namely,
$\PP^1$ minus a set of $n\ge 3$ points. A Landau-Ginzburg model mirror
to $C$ can be constructed by viewing $C$ as a hypersurface in $(\C^*)^2$
(which can be compactified to a rational curve in $\PP^1\times\PP^1$ or a
Hirzebruch surface). The procedure described in \cite{KKOY}
(or those in \cite{GKR} or \cite{AAK}) then yields
a (noncompact) toric 3-fold $X(n),$ together with a superpotential
$W:X(n)\to\C,$ which we take as the mirror to $C$.
For $n=3$ the Landau-Ginzburg model
$(X(3), W)$ is the three-dimensional
affine space $\C^3$ with the superpotential $W=xyz$, while for $n>3$ points
$X(n)$ is more complicated (it is a toric resolution of a 3-dimensional
singular affine toric variety); see Section \ref{s:LGmodel} and Fig.~\ref{fig:LG}
for details.

We focus on one side of homological mirror symmetry, in which we consider the  wrapped Fukaya category
of $C$ (as defined in \cite{AS,generate}), and the associated triangulated derived category  $D \mathcal{W}(C)$ (see Section (3j) of \cite{seidel-book}).  Our main theorem asserts that this triangulated category is equivalent to the triangulated category of singularities \cite{Orlov} of the singular fiber $W^{-1}(0)$ of $(X(n),W)$. In fact, we obtain a slightly stronger result than stated below, namely a quasi-equivalence between the natural $A_\infty$-enhancements of these two
categories.

\begin{theo}\label{th:main}
Let $C$ be the complement of a finite set of $n\ge 3$ points in $\PP^1,$ and
let $(X(n),W)$ be the Landau-Ginzburg model defined in Section \ref{s:LGmodel}.
Then the derived wrapped Fukaya category
 of $C$, $D\mathcal{W}(C)$,  is equivalent to the triangulated category of singularities
$D_{sg}(W^{-1}(0)).$
\end{theo}

\noindent
The other side of homological mirror symmetry is generally considered to be out of reach of current
technology for these examples, due to the singular nature of the critical locus of~$W$.

\begin{remark}{\rm
The case $n=0$ falls under the rubric of mirror symmetry for Fano varieties,
and is easy to prove since the equatorial circle in $S^2$ is the unique non-displaceable Lagrangian, and the mirror superpotential has exactly one non-degenerate isolated singularity.
Mirror symmetry for $\C$ is trivial in this direction since the Fukaya category
completely vanishes in this case, and the mirror superpotential has no critical point.
Finally, the case $n=2$ can be recovered as a degenerate case of our analysis,
but was already essentially known to experts because the cylinder is symplectomorphic
to the cotangent bundle of the circle, and Fukaya categories of cotangent bundles admit
quite explicit descriptions using string topology \cite{FSS,A-cotangent}.}
\end{remark}

The general strategy of proof is similar to that used by Seidel for the
genus 2 curve, and inspired by it. Namely, we identify specific generators
of the respective categories
(in Section~\ref{s:wrapped} for $D\mathcal{W}(C)$, using a generation
result proved in Appendix A, and in
Section~\ref{s:Dbsing} for $D_{sg}(W^{-1}(0))$,  and show that the corresponding $A_\infty$ subcategories on either side are equivalent by appealing to an algebraic classification lemma  (Section \ref{s:classif}); see also Remark
\ref{rmk:generation} for more about generation).  A general result due to Keller (see Theorem 3.8 of \cite{Ke} or Lemma 3.34 of \cite{seidel-book}) implies that the categories $D\mathcal{W}(C)$ and $D_{sg}(W^{-1}(0))$ are therefore equivalent to the derived categories of the same $A_{\infty}$ category, hence are equivalent to each other.

\medskip

This strategy of proof can be extended to higher genus punctured Riemann
surfaces, the main difference being that one needs to consider larger
sets of generating objects (which in the general case leads to a slightly
more technically involved argument). However, there is a special case in
which the generalization of our result is particularly straightforward,
namely the case of unramified cyclic covers of punctured spheres.
The idea that Fukaya categories of unramified covers are closely related
to those of the base is already present in Seidel's work \cite{Segenus2} and the argument we use is again very similar (this approach can be used in higher dimensions as well, as evidenced in Sheridan's work \cite{sheridan}).  As an illustration,  we prove the following result in Section \ref{s:covers}:

\begin{theo}\label{th:cover}
Given an unramified cyclic $D$-fold cover $C$ of\/ $\PP^1-\{3\ \text{points}\},$
there exists an action of $G=\Z/D$ on the Landau-Ginzburg model $(X(3),W)$
such that the derived wrapped Fukaya category $D\mathcal{W}(C)$
is equivalent to the equivariant triangulated category of
singularities $D_{sg}^G(W^{-1}(0)).$
\end{theo}

\begin{remark}
{\rm
The main difference between our approach and that developed in Seidel and
Sheridan's papers \cite{Segenus2,sheridan} is that, rather than
compact (possibly immersed) Lagrangians, we consider the
wrapped Fukaya category, which is strictly larger. The Floer homology of the immersed
closed curve considered by Seidel in \cite{Segenus2} can be recovered from our calculations,
but not vice-versa. There is an obvious motivation for restricting to
that particular object (and its higher dimensional analogue \cite{sheridan}): even though it does not determine the entire
$A$-model in the open case, it gives access to the Fukaya category
of closed Riemann surfaces or projective Fermat hypersurfaces in a fairly direct
manner. On the other hand, open Riemann surfaces and other
exact symplectic manifolds are interesting both in themselves and as building
blocks of more complicated manifolds.
}
\end{remark}

We end this introduction with a brief outline of this paper's organization:  Section \ref{s:Ainfinity} explicitly defines a category $A$, and introduces rudiments of Homological Algebra which are used, in the subsequent section, to classify $A_{\infty}$ structures on this category up to equivalence.  Section \ref{s:wrapped} proves that $A$ is equivalent to a cohomological subcategory of the wrapped Fukaya category of a punctured sphere, and uses the classification result to identify the $A_{\infty}$ structure induced by the count of homolorphic curves.  In this Section, we also prove that our distinguished collection of objects strongly generate the wrapped Fukaya category.

The mirror superpotential is described in Section \ref{s:LGmodel}, and a collection of sheaves whose endomorphism algebra in the category of matrix factorizations is isomorphic to $A$ is identified in the next Section, in which the $A_{\infty}$ structure coming from the natural $dg$ enhancement is also computed and a generation statement proved.  At this stage, all the results needed for the proof of Theorem \ref{th:main} are in place.  Section \ref{s:covers} completes the main part of the paper by constructing the various categories appearing in the statement of Theorem \ref{th:cover}.  The paper ends with two appendices; the first proves a general result providing strict generators for wrapped Fukaya categories of curves, and the second shows that the categories of singularities that we study are idempotent complete.

\subsection*{Acknowledgements}
M.\,A.\ was supported by a Clay Research Fellowship.
D.\,A.\ was partially supported by NSF grants DMS-0652630 and DMS-1007177.
A.\,E.\ was partially supported by the Dynasty Foundation, 
NSh grant 4713.2010.1, and by AG Laboratory HSE, RF government grant, ag. 11.G34.31.0023.
L.\,K.\ was funded by NSF grant DMS-0600800, NSF FRG grant DMS-0652633, FWF
grant P20778 and an ERC grant~-- GEMIS.
D.\,O.\ was partially supported by  RFBR grants 10-01-93113, 11-01-00336, 11-01-00568, NSh grant 4713.2010.1, and
by AG Laboratory HSE, RF government grant, ag. 11.G34.31.0023.
D.\,O.\ thanks the Simons Center for Geometry and Physics for its hospitality and stimulating atmosphere.

The authors acknowledge the hospitality of ESI (Vienna) and MSRI (Berkeley),
where some of this work was carried out.
We would also like to thank Paul Seidel for helpful discussions.

\section{$A_{\infty}\text{-}$structures} \label{s:Ainfinity}

Let $\cA$ be a small $\Z\text{-}$graded category over a field $\kk,$
i.e.\ the morphism spaces $\cA(X, Y)$ are $\Z\text{-}$graded $\kk$-modules and the compositions
$$
\cA(Y, Z)\otimes\cA(X, Y)\longrightarrow \cA(X, Z)
$$
are morphisms of $\Z\text{-}$graded $\kk$-modules.
By grading we will always mean $\Z\text{-}$gradings.

By an
$A_{\infty}\text{-}$structure on $\cA$ we mean a collection of graded maps
$$
m_k:\cA(X_{k-1},X_k)\otimes\dots\otimes\cA(X_0,X_1)\longrightarrow
\cA(X_0,X_k),\quad X_i\in \cA,\quad k\geq 1
$$
of degree
$\deg(m_k)=2-k,$ with $m_1=0,$ and $m_2$ equal to the usual
composition in $\cA,$ such that all together they define an
$A_{\infty}\text{-}$category, i.e.\ they satisfy the
$A_{\infty}\text{-}$associativity equations
\begin{equation}\label{ainf}
\sum_{\substack{s,l,t\\ s+l+t=k}} (-1)^{s+lt}
m_{k-l+1}(\id^{\otimes s}\otimes m_l\otimes \id^{\otimes t})=0,
\end{equation}
for all $k\ge 1.$
Note that additional signs appear when these formulas are applied to
elements, according to the Koszul sign rule $(f\otimes g)(x\otimes y)= (-1)^{\deg
g\cdot\deg x}f(x)\otimes g(y)$ (see \cite{Keller, seidel-book}).

Two $A_{\infty}\text{-}$structures $m$ and $m'$ on $\cA$ are said to be
{\em strictly homotopic} if there exists an $A_{\infty}\text{-}$functor $f$
from $(\cA,m)$ to $(\cA,m')$ that acts identically on objects
and for which $f_1=\id$ as well.

We also recall that an $A_{\infty}\text{-}$functor $f$ consists of a map
$\bar{f}: \Ob(\cA,m)\to \Ob(\cA,m')$  and graded maps
$$
f_k: \cA(X_{k-1},X_k)\otimes\dots\otimes\cA(X_0,X_1)\longrightarrow
\cA(\bar{f}X_0,\bar{f}X_k),\quad  X_i\in \cA,\quad k\geq 1
$$
of degree $1-k$ which satisfy the equations
\begin{multline}\label{ainff}
\sum_{r}\sum_{\substack{u_1,\dots, u_r\\ u_1+\cdots+u_r=k}} (-1)^{\varepsilon} m'_r(f_{u_1}\otimes\dots\otimes f_{u_r})
=\sum_{\substack{ s,l,t,\\ s+l+t=k}} (-1)^{s+lt} f_{k-l+1}(\id^{\otimes s}\otimes
m_l\otimes\id^{\otimes t}),
\end{multline}
where the sign on the left hand side is given by
$\varepsilon=(r-1)(u_1-1)+(r-2)(u_{2}-1)+\cdots+ i_{r-1}.$

Now
we introduce a $\kk$-linear category $A$ that plays a central role
in our considerations. It depends on an integer $n\ge 3$ and is
defined by the following rule:
\begin{equation}
\label{eq:defA}
\Ob(A)=\{X_1,\dots,X_n\},\quad A(X_i,X_j)=\begin{cases}\kk[x_i, y_{i}]/(x_i y_{i}) & \text{for }j=i,\\
\kk[x_{i+1}]\,u_{i,i+1}= u_{i,i+1}\,\kk[y_{i}] & \text{for }j=i+1,\\
\kk[y_{i-1}]\,v_{i,i-1}= v_{i,i-1}\,\kk[x_i] & \text{for }j=i-1,\\
0 &\text{otherwise.}\end{cases}
\end{equation}
Here the indices are mod $n,$ i.e. we put $X_{n+1}=X_1,$
and $x_{n+1}=x_1, \, y_{n+1}=y_1.$

Compositions in this category are defined as follows. First of all, the
above formulas already define $A(X_i,X_i)$ as $\kk$-algebras,
and $A(X_i,X_{j})$ as  $A(X_i,X_i)\text{--}A(X_{j},X_{j})\text{-}$bimodules.
To complete the definition, we set
$$
(x_i^k u_{i-1,i})\circ (v_{i,i-1}x_i^l):=x_i^{k+l+1},
\quad (v_{i,i-1}x_i^l)\circ (x_i^k u_{i-1,i}):=y_{i-1}^{k+l+1}
$$
for any two morphisms
$x_i^k\,u_{i-1,i}\in A(X_{i-1},X_i)$  and $v_{i,i-1}x_i^l \in A(X_i,X_{i-1}).$
All the other compositions vanish. Thus, $A$ is defined as a $\kk$-linear category.

Choosing some collection of odd integers $p_1,\dots,p_n,q_1,\dots,q_n$ we can define a grading on $A$  by the formulas
$$
\deg(u_{i-1,i}:X_{i-1}\longrightarrow X_i):=p_i,\quad\deg(v_{i,i-1}:X_i\longrightarrow X_{i-1}):=q_i.
$$
That implies  $\deg x_i=\deg y_{i-1}=p_i+q_i.$ All these gradings are refinements of the same $\Z/2$-grading on $A.$

In what follows we will require that the following conditions hold:
\begin{equation}\label{cond}
p_1,\dots, p_n, q_1,\dots, q_n \quad\text{are odd, and }\quad p_1+\dots+p_n=q_1+\dots+q_n=n-2.
\end{equation}
\begin{defi}\label{categ}
For such collections of $p=\{p_i\}$ and $q=\{q_i\}$ we denote by
$A_{(p,q)}$ the corresponding $\Z$-graded category. \end{defi}

We are interested in describing all $A_\infty$-structures on the category $A_{(p,q)}.$ As we will see, these structures are in bijection
with pairs $(a, b)$ of elements $a, b\in \kk.$

Let $\cA$ be a small $\Z$-graded category over a field $\kk.$
It will be convenient to consider the bigraded Hochschild complex $CC^{\bdot}(\cA)^{\bdot},$
$$
CC^{k+l}(\cA)^l=\prod_{\substack{X_0,\dots,X_k\in
\cA}}\Hom^l(\cA(X_{k-1},X_k)\otimes\dots\otimes\cA(X_0,X_1),\;\cA(X_0,X_k)).
$$
with the Hochschild differential $d$ of bidegree $(1,0)$ defined by
\begin{multline*}
dT(a_{k+1},\dots, a_1)= (-1)^{(k+l)(\deg(a_1)-1)+1} T(a_{k+1},\dots,
a_2)a_1 +\\
+\sum_{j=1}^k (-1)^{\epsilon_j+(k+l)-1} T(a_{k+1},\dots,
a_{j+1}a_j,\dots, a_1) + (-1)^{\epsilon_k +(k+l)}
a_{k+1}T(a_k,\dots, a_1),
\end{multline*}
where the sign is defined by the rule
$\epsilon_j=\sum_{i=1}^j\deg{a_i}-j.$
We denote by $HH^{k+l}(\cA)^l$ the bigraded Hochschild cohomology.

%
%
%
%
%


Denote by $A_{\infty}S(\cA)$ the set of
$A_{\infty}\text{-}$structures on $\cA$
up to strict homotopy.

Basic obstruction theory implies the following proposition, which will be sufficient for our purposes.

\begin{prop}\label{MC_solutions} Assume that the small\/ $\Z$-graded $\kk$-linear category $\cA$ satisfies the following conditions
\begin{equation}\label{no_obstr1}
HH^2(\cA)^j=0\text{ for } j\le -1 \text{ and } j\ne -l,
\end{equation}
and
\begin{equation}\label{no_obstr}
HH^3(\cA)^j=0\text{ for }j<-l,
\end{equation}
for some positive integer\/ $l\ge 1.$
Then for any $\phi\in HH^{2}(\cA)^{-l}$ there is an
$A_{\infty}\text{-}$structure $m^{\phi}$ with $m_3=\cdots=m_{l+1}=0,$ for which the class of $m_{l+2}$ in $HH^2(\cA)^{-l}$ is equal to
$\phi.$ Moreover, the natural map
$$
HH^2(\cA)^{-l}\to A_{\infty}S(\cA), \quad \phi\mapsto m_{\phi},
$$
is a surjection, i.e. any other $A_{\infty}\text{-}$structure is strictly homotopic to $m_{\phi}.$
\end{prop}

To prove this proposition, we recall some well-known statements from obstruction theory.
Let $m$ be an $A_{\infty}\text{-}$structure on a graded category $\cA.$ Let us consider the
$A_{\infty}\text{-}$constraint (\ref{ainf}) of order $k+1.$ Since $m_1=0$ it is the first constraint that
involves $m_k.$ Moreover, it can be written in the form
\begin{equation}
d m_k=\Phi_k(m_3,\dots,m_{k-1}),
\end{equation}
where $d$ is the Hochschild differential and
$\Phi_k=\Phi_k(m_3,\dots,m_{k-1})$ is a quadratic expression.

Similarly, let $m$ and $m'$ be two $A_{\infty}$-structures on a graded category $\cA,$ and let
$f=(\bar{f}=\id;f_1=\id,f_2, f_3,\dots )$ be a strict
homotopy between $m$ and $m'.$ Since $m_1=m'_1=0$ the order $k+1$
$A_{\infty}\text{-}$constraint (\ref{ainff}) is the first one that contains $f_k.$ It can be written as
\begin{multline}\label{Psi}
df_k=\Psi_k(f_2,\dots,f_{k-1};m_3,\dots, m_{k+1}; m'_3,\dots, m'_{k+1})=\\=\Psi'_k(f_2,\dots, f_{k-1}; m_3,\dots, m_k; m'_3,\dots, m'_{k})+ m'_{k+1}-m_{k+1}
\end{multline}
where $d$ is the Hochschild differential and
$\Psi_k$ is a polynomial expression.
The following lemma is well-known and can be proved by a direct calculation.

\begin{lemma}\label{obstr_m_n} In the above notations, let $d$ be the Hochschild differential.
\begin{itemize}
\item[(1)] Assume that the first $k$ $A_{\infty}\text{-}$constraints (\ref{ainf}), which depend only on $m_{<k},$ hold.
 Then $$d \Phi_k(m_3,\dots, m_{k-1})=0.$$
\item[(2)] Let $m$ and $m'$ be two $A_{\infty}\text{-}$structures on a graded category $\cA,$ and $f$ a strict homotopy
between them.
Assume that the first $k$ $A_{\infty}\text{-}$constraints (\ref{ainff}), which depend
only on $f_{<k},$ hold. Then
$$
d\Psi_k (f_2,\dots, f_{k-1};m,m')=0.
$$
\end{itemize}
\end{lemma}

The following lemma is a direct consequence of the $k^{th}$ $A_{\infty}\text{-}$constraint (\ref{ainff}).
\begin{lemma}\label{perturb_m_n}Let $m$ and $m'$ be two $A_{\infty}$-structures on a graded category $\cA.$
 Let $f:(\cA,m)\to (\cA,m')$ be an
$A_{\infty}\text{-}$homomorphism with $f_1=\id,$ and $f_i=0$ for
$1<i< k-1.$ Then $m_i=m_i'$ for $i<k$ and $d f_{k-1}=m'_k-m_k.$
\end{lemma}

\begin{proof}[Proof of Proposition \ref{MC_solutions}.]
 We define the desired surjection as follows. Let $\phi\in HH^2(\cA)^{-l}$ be some class,
and $\widetilde{\phi}\in CC^2(\cA)^{-l}$ its representative. Consider the partial $A_{\infty}$-structure $(m_3,\dots,m_{l+2})$ with
$$
m_{l+2}=\widetilde{\phi},\quad m_3=\dots=m_{l+1}=0.
$$
The maps $m_{\leq l+2}$ satisfy all the required equations (\ref{ainf})
which do not involve $m_{> l+2}$ (there is only one nontrivial such equation, $d m_{l+2}=0$).
By induction on $k,$ the equation
$$
d m_k=\Phi_k(m_3,\dots,m_{k-1})
$$
has a solution for each $k> l+2,$ since
we know from part (1) of Lemma \ref{obstr_m_n} that $d\Phi_k=0$  and from condition \eqref{no_obstr} that $HH^3(\cA)^j=0$ when $j<-l.$
This means that $(m_3,\dots,m_{l+2})$ lifts to some $A_{\infty}$-structure $m^{\widetilde{\phi}}$ on $\cA.$

Moreover, by condition \eqref{no_obstr1} we have  $HH^2(\cA)^{j}=0$ when $j<-l,$ and by Lemma~\ref{obstr_m_n}\,(2) we know that $d\Psi_k=0.$
This implies
 that the equation \eqref{Psi} can be solved for all $k>l+1,$ i.e.\ the lift
is unique up to strict homotopy. Finally, similar considerations and  Lemma \ref{perturb_m_n} give that the resulting element
$m^{\widetilde{\phi}}\in A_{\infty}S(\cA)$ depends only on $\phi,$ not on $\widetilde{\phi}.$

Therefore, the map $HH^2(\cA)^{-l}\to A_{\infty}S(\cA)$ is well-defined. Now we show that it is surjective. Let us consider an $A_{\infty}\text{-}$structure
$m'$ on $\cA$ and  let us take some $A_{\infty}\text{-}$structure $m^{\widetilde{\phi}}$
with $m_3=\dots=m_{l+1}=0$ and $m_{l+2}=\widetilde{\phi}$ as above.
By condition \eqref{no_obstr1} $HH^{2}(\cA)^{j}=0$ for all $j\le -1$ and $ j\ne l.$ Hence by (2)
of Lemma \ref{obstr_m_n} we can construct a strict homotopy $f$
between $m'$ and $m^{\widetilde{\phi}}$ if and only if the expression $\Psi_{l+1}$ from (\ref{Psi})
is exact. Since $\Psi_{l+1}$ depends linearly on $m_{l+2},$
we can find $\widetilde{\phi}$ such that the class of $\Psi_{l+1}$ in the cohomology group $HH^2(\cA)^{-l}$
vanishes; hence, for this choice of $\widetilde{\phi},$ the $A_\infty$-structure $m'$ will be strictly homotopic to $m^{\widetilde{\phi}}.$
This completes the proof of the proposition.\end{proof}

\section{A classification of $A_{\infty}\text{-}$structures}\label{s:classif}

In this section we describe all $A_\infty\text{-}$structures on the category $A_{(p,q)}.$
The main technical result of this section is the following proposition:

\begin{prop}\label{classification}
Let  $A$ be the category with $n\ge 3$ objects defined by \eqref{eq:defA}. Then

\begin{itemize}
\item[(1)] For any two elements $a,b\in \kk,$ there exists a $\Z/2\text{-}$graded $A_{\infty}\text{-}$structure
$m^{a,b}$ on $A,$ compatible with
all\/ $\Z\text{-}$gradings satisfying \eqref{cond}, such that $m^{a,b}_3=\cdots=m^{a,b}_{n-1}=0$ and
$$
\qquad m^{a,b}_n(u_{i-1,i},u_{i-2,i-1},\dots,u_{i,i+1})(0)=a,\quad
m^{a,b}_n(v_{i+1,i},v_{i+2,i+1},\dots,v_{i,i-1})(0)=b
$$
for any $1\le i\le n,$ where $\cdot(0)$ means the constant
coefficient of an element of $A(X_i,X_i),$ i.e.\ the coefficient of\/ $\id_{X_i}.$
\item[(2)] Moreover, for any $\Z\text{-}$grading $A_{(p,q)}$ where $(p,q)$ satisfy \eqref{cond}, the map
$$
\kk^2\to A_{\infty}S(A_{(p,q)}),\quad (a,b)\mapsto m^{a,b},
$$
is a bijection, i.e.\ any $A_{\infty}\text{-}$structure $m$ on $A_{(p,q)}$ is strictly homotopic to $m^{a,b}$ with
\begin{align}\label{a,b}
a=m_n(u_{n,1},u_{n-1,n},\dots,u_{1,2})(0),\quad
b=m_n(v_{2,1},v_{3,2},\dots,v_{1,n})(0).
\end{align}
\end{itemize}
\end{prop}

The proof of this proposition essentially reduces to the computation of the Hochschild cohomology of $A_{(p,q)}.$

\begin{lemma}\label{Hoch_comp}Let $A_{(p,q)}$ be the $\Z\text{-}$graded category with $n\ge 3$
objects as in Definition \ref{categ}. Then the bigraded Hochschild
cohomology of $A_{(p,q)}$ is the following:
$$
HH^{d}(A_{(p,q)})^j\cong\begin{cases}
\kk^2 & \text{ for each }d\ge 2\quad \text{when}\quad  j=\left\lfloor\dfrac{d}2\right\rfloor(2-n),\\
0 & \text{ in all other cases when }\quad d-j\ge 2.
\end{cases}
$$
\end{lemma}

\begin{proof}We have a subcomplex
\begin{equation}
\label{red}
CC^{\bdot}_{red}(\cA)^{\bdot}\subset CC^{\bdot}(\cA)^{\bdot},
\end{equation}
the so-called reduced Hochschild complex, which consists of cochains
that vanish on any sequence of morphisms containing some identity
morphism. It is classically known that the inclusion \eqref{red} is
a quasi-isomorphism. We will compute Hochschild cohomology using the
reduced Hochschild complex. For convenience, we will write just $A$
instead of $A_{(p,q)}.$ Let
$$
\widetilde{A}=\bigoplus_{\substack{i,j}}A(X_i,X_j).
$$
This is a graded algebra. We have a non-unital graded algebra
$$
A_{red}:=\ker\Big(\bigoplus_{\substack{i,j}}A(X_i,X_j)\to\bigoplus_{\substack{i}}\kk\cdot\id_{X_i}\Big).
$$

Let $R=\bigoplus_{\substack{i}}\kk\cdot\id_{X_i}.$ Then both $A_{red}$ and $\widetilde{A}$ are $R\text{-}R\text{-}$bimodules, and
$$
CC^{k+l}_{red}(A)^l=\Hom_{R-R}^l({A_{red}}^{{\otimes\!}_R k},\widetilde{A}),\quad  k\ge 0.
$$

Denote by $A_i\subset A_{red}$ the subalgebra generated by $u_{i-1,i}$ and $v_{i,i-1}.$ Then we have an isomorphism
$$
A_{red}\cong\bigoplus_{\substack{i}}A_i
$$
of non-unital graded algebras (because $A_i\cdot A_j=0$ for $i\ne
j$).

Consider the bar complex of $R\text{--}R\text{-}$bimodules
$$
K_i^{\bdot}=\overline{T}(s A_i)=\bigoplus_{m>0} (s A_i)^{{\otimes\!}_R
m},
$$
where $(sA_i)^p=(A_i)^{p+1}$
and  the differential is
the bar differential
$$
D(sa_k\otimes\cdots\otimes sa_{1})=\sum_{i=1}^{k-1}  (-1)^{\epsilon_i} sa_k\otimes\cdots\otimes sa_{i+1} a_{i}\otimes\cdots sa_1
$$
with $\epsilon_i=\sum_{j\le i}\deg sa_j.$

Denote by $A_i(d)\subset A_i,$ $d>0,$ the 2-dimensional subspace generated by
the two products of $u_{i-1,i}$ and $v_{i,i-1}$ of length $d,$ i.e.\ $A_i(2m+1)$ is generated by $x_i^m u_{i-1,i}$ and $v_{i,i-1} x_i^m$
while $A_i(2m)$ is generated by $x_i^m$ and $y_{i-1}^{ m}.$
Consider the subcomplex
$$
K_i^{\bdot}(d)\subset K_i^{\bdot},\quad K_i^{\bdot}(d)=\bigoplus_{\substack{d_1+\dots+d_l=d,\\
l>0}}sA_i(d_1)\otimes_R sA_i(d_2)\otimes_R\dots\otimes_R sA_i(d_l).
$$

\begin{lemma}\label{trivial}$K_i(1)\cong sA_i(1),$ and for $d>1$ the complex $K_i^{\bdot}(d)$ is acyclic.\end{lemma}
\begin{proof}
The result is obvious for $d=1.$ For $d\ge 2,$
we subdivide the complex $K_i(d)$ into two parts, according to whether $d_l=1$ or $d_l>1.$ The first part is
$K_i(d-1)\otimes_R sA_i(1).$
We also note that the product map $A_i(d_l-1)\otimes_R A_i(1)\to A_i(d_l)$ is an isomorphism. Hence the second part of
the complex is isomorphic to
$K_i(d-1)\otimes_R A_i(1).$ Using these identifications, we conclude that
$K_i(d)$ is isomorphic to the total complex of the bicomplex $K_i(d-1)\otimes_R A_i(1)\to K_i(d-1)\otimes_R A_i(1),$ where the connecting map
is the identity map. It is therefore acyclic.
\end{proof}

Now let
$$
K^{\bdot}= T(s A_{red})=\bigoplus_{\substack{m\ge 0}}(sA_{red})^{{\otimes\!}_R m}.
$$
We have an isomorphism of graded vector spaces
$$
K^{\bdot}=R\oplus \bigoplus_{\substack{w> 0\\ i_t\ne i_{t+1}}}K_{i_1}^{\bdot}\otimes_R\dots\otimes_R K_{i_w}^{\bdot},
$$
which is also an isomorphism of complexes because $A_i\cdot A_j=0$ for $i\ne j.$ Define subcomplexes
$$
 K^{\bdot}(0)=R,\quad K^{\bdot}(d)=\bigoplus_{\substack{w>0\\
d_1+\dots+d_w=d,\\i_t\ne i_{t+1}}}K_{i_1}^{\bdot}(d_1)\otimes_R\dots\otimes_R K_{i_w}^{\bdot}(d_w)\ \text{ for }\ d\ge 1.
$$
Consider the full decreasing filtration
$$
CC^{\bdot}(A)^{\bdot}_{red}=L_1^{\bdot}(A)^{\bdot}\supset L_2^{\bdot}(A)^{\bdot}\supset\dots,
$$
where $L_r^{\bdot}(A)^{\bdot}$ consists of all cochains vanishing on $K^{\bdot}(i)$ for $0\le i<r.$

Denote by $\Gr_r^{\bdot}(A)^{\bdot}=L_r^{\bdot}(A)^{\bdot}/L_{r+1}^{\bdot}(A)^{\bdot}$
the associated graded factors of this filtration.
The Hochschild differential $d$ induces a differential
$$
d_0: \Gr_r^{\bdot}(A)^{\bdot}\to
\Gr_r^{\bdot+1}(A)^{\bdot}.
$$
It is easy to see that $d_0$ coincides with a differential defined by the bar differential $D$ on $K^{\bdot}.$
Therefore, Lemma \ref{trivial} implies that for $r\ge 1$ we have
\begin{multline*}
H^{r+j}(\Gr_r^{\bdot}(A)^{\bdot})^j=\\
\Hom^j_{R-R}\Big(\bigoplus_{\substack{t\in \Z/n\Z}}\kk\cdot u_{t-1,t}\otimes\dots\otimes u_{t-r,t-r+1}
\ \oplus \bigoplus_{\substack{t\in \Z/n\Z}}\kk\cdot v_{t+1,t}\otimes\dots\otimes v_{t+r,t+r-1},\;\widetilde{A} \Big),
\end{multline*}
 and
\begin{align*}
&H^{i+j}(\Gr_r^{\bdot}(A)^{\bdot})^j=0\quad \text{ for }\quad i\ne r.
\end{align*}

The first differential
$$
d_1:H^{r+j}(\Gr_r^{\bdot}(A)^{\bdot})^j\to H^{r+j+1}(\Gr_{r+1}^{\bdot}(A)^{\bdot})^j
$$
in the spectral sequence $E_1^{r, j}=H^{r+j}(\Gr_r^{\bdot}(A)^{\bdot})^j$
is given by the formula
$$\begin{cases}\ d_1\phi(u_{t-1,t},u_{t-2,t-1},\dots,u_{t-r-1,t-r})=&\pm u_{t-1,t}\phi(u_{t-2,t-1},\dots,u_{t-r-1,t-r})\\
&\pm\phi(u_{t-1,t},\dots,u_{t-r,t-r+1})u_{t-r-1,t-r},\\
\ d_1\phi(v_{t+1,t},v_{t+2,t+1},\dots,v_{t+r+1,t+r})\,=&\pm v_{t+1,t}\phi(v_{t+2,t+1},\dots,v_{t+r+1,t+r})\\
&\pm \phi(v_{t+1,t},\dots,v_{t+r,t+r-1})v_{t+r+1,t+r}.\end{cases}
$$
It is clear that $H^{r+j}(\Gr_r^{\bdot}(A)^{\bdot})^j\ne 0$ only for
$r\equiv 0,\pm 1\text{ mod }n$
and the spectral sequence $(E_1^{\bdot,\bdot}, d_1)$ consists of the following simple complexes
\begin{multline}\label{e1}
0\to H^{mn+j-1}(\Gr_{mn-1}^{\bdot}(A)^{\bdot})^j
\to H^{mn+j}(\Gr_{mn}^{\bdot}(A)^{\bdot})^j\\
\to H^{mn+j+1}(\Gr_{mn+1}^{\bdot}(A)^{\bdot})^j
\to 0.
\end{multline}

Let $m>0.$ 
%
%
%
%
Now, 
if $j\ne m(2-n),$ then the complexes \eqref{e1} are acyclic. If $j=m(2-n),$
then the complex \eqref{e1} 
has only two nontrivial terms and is
 \begin{multline*}\label{explicit}
0\to \Hom_{R-R}\Big(\bigoplus_{\substack{t\in \Z/n\Z}}\kk\cdot u_{t-1,t}\otimes\dots\otimes u_{t-mn,t-mn+1}\oplus
 \bigoplus_{\substack{t\in \Z/n\Z}}\kk\cdot v_{t+1,t}\otimes\dots\otimes v_{t+mn,t+mn-1},R\Big)\\
 \to
\Hom_{R-R}\Big(\bigoplus_{\substack{t\in \Z/n\Z}}\kk\cdot u_{t-1,t}\otimes\dots\otimes u_{t-mn-1,t-mn}\oplus
 \bigoplus_{\substack{t\in \Z/n\Z}}\kk\cdot v_{t+1,t}\otimes\dots\otimes v_{t+mn+1,t+mn},\\\bigoplus_{\substack{t\in \Z/n\Z}}\kk\cdot u_{t-1,t}
 \oplus \bigoplus_{\substack{t\in \Z/n\Z}}\kk\cdot v_{t+1,t}\Big)  \to 0.
 \end{multline*}

Thus, the computation of the cohomology of $d_1$ reduces to an easy computation of the kernel and the cokernel of this map.
For $m>0$ we obtain that the cohomology of $d_1$ is the following:
\bigskip

\begin{tabular}{ll}
$H^{2m}(E_1^{\bdot,\bdot},d_1)^{m(2-n)}\cong \kk^2,$ & $\phi^{a,b}(u_{i-1,i},u_{i-2,i-1},\dots,u_{i,i+1})=a\cdot\id_{X_i},$\\
& $\phi^{a,b}(v_{i+1,i},v_{i+2,i+1},\dots,v_{i,i-1})=b\cdot\id_{X_i},\quad a,b\in \kk,$\\
$H^{2m+1}(E_1^{\bdot,\bdot},d_1)^{m(2-n)}\cong \kk^2,$ & $\psi^{c,d}(u_{i-1,i},u_{i-2,i-1},\dots,u_{i-1,i})=\delta_{i1}\cdot c\cdot u_{i-1,i},$\\
& $\psi^{c,d}(v_{i+1,i},v_{i+2,i+1},\dots,v_{i+1,i})=\delta_{i1}\cdot d\cdot v_{i+1,i},\;\; c,d\in \kk,$\\
$H^{i+j}(E_1^{\bdot, \bdot},d_1)^j=0$ & in all other cases with   $i\geq 2.$
\end{tabular}
\bigskip

It is easy to see that the spectral sequence degenerates at the $E_2^{\bdot,\bdot}$ term, i.e.\ all these classes
can be lifted to actual
Hochschild cohomology classes. This proves Lemma \ref{Hoch_comp}.
\end{proof}

\begin{proof}[Proof of Proposition \ref{classification}]
Part (1) directly follows from Lemma \ref{Hoch_comp} and Proposition \ref{MC_solutions}.

Lemma \ref{Hoch_comp} and Proposition \ref{MC_solutions} also imply that the map $(a,b)\mapsto m^{a,b}$ is
a surjection on $A_{\infty}S(A_{p,q}).$
Further, it is straightforward to check that the coefficients \eqref{a,b} are invariant under strict homotopy.
This proves part (2) of the proposition.\end{proof}

\begin{remark}\label{rm:structures}{\rm
Note that autoequivalences of the graded category
$A_{p,q}$ act on the set of $A_{\infty}\text{-}$structures $A_{\infty}S(A_{p,q}).$
In particular, it is easy to see that all $A_{\infty}\text{-}$structures $m^{a,b}$ with $a\ne 0,\; b\ne 0$ yield equivalent $A_{\infty}\text{-}$categories,
all of them quasi-equivalent to $m^{1,1}.$
We also have three degenerate $A_{\infty}\text{-}$categories defined by $m^{0,1}, m^{1,0}$ and $m^{0,0},$ where the last-mentioned
 coincides with the category $A_{p,q}$ itself.
}
\end{remark}

\section{The wrapped Fukaya category of $C$}
\label{s:wrapped}

In this section we study the wrapped Fukaya category of $C.$ Recall that the
wrapped Fukaya category of an exact symplectic manifold (equipped with a
Liouville structure) is an $A_\infty$-category whose objects are (graded)
exact Lagrangian submanifolds which are invariant
under the Liouville flow outside of a compact subset. Morphisms and
compositions are defined by considering Lagrangian Floer intersection theory
perturbed by the flow generated by a Hamiltonian function $H$ which is
quadratic at infinity. Specifically, the wrapped Floer complex
$\Hom(L,L')=CW^*(L,L')$ is generated by time 1 trajectories of the
Hamiltonian vector field $X_H$ which connect $L$ to $L',$ or equivalently,
by points in $\phi_H^1(L)\cap L'$; compositions count solutions to a
perturbed Cauchy-Riemann equation. In the specific case of punctured
spheres, these notions will be clarified over the course of the discussion;
the reader is referred to \cite[Sections 2--4]{generate} for a complete
definition (see also \cite{AS} for a different construction).

The goal of this section is to prove the following:

\begin{theo}\label{th:fukaya}
The wrapped Fukaya category of $C$ $($the complement of $n\ge 3$ points
in~$\PP^1)$ is strictly generated by $n$ objects $L_1,\dots,L_n$ such that
$$\bigoplus_{i,j} \Hom(L_i,L_j)\simeq \bigoplus_{i,j} A(X_i,X_j),$$
where $A$ is the category defined in \eqref{eq:defA} $($with any grading
satisfying \eqref{cond}$),$ and the associated
$A_\infty\text{-}$structure is strictly homotopic to $m^{1,1}.$
\end{theo}

We now make a couple of remarks in order to clarify the meaning of this
statement.

\begin{remark}\label{rmk:generation}
{\rm (1) A given set of objects is usually said to generate a
triangulated category $\mathcal{T}$ when the smallest triangulated
subcategory of $\mathcal{T}$ containing the given objects and closed
under taking direct summands is the whole category $\mathcal{T}$; or
equivalently, when every object of $\mathcal{T}$ is
isomorphic to a direct summand of a complex built out of the
given objects. In the symplectic geometry literature this concept is
sometimes called ``split-generation'' (cf.\ e.g.~\cite{generate}).
By contrast, in this paper we always consider a stronger notion of
generation, in which direct summands are not allowed: namely, we say
that $\mathcal{T}$ is {\em strictly generated} by the given objects if
the minimal triangulated subcategory containing
these objects is $\mathcal{T}.$}

{\rm (2) The $A_\infty$-category $\W(C)$ is not triangulated, however it
admits a natural triangulated enlargement, the $A_\infty$-category of
{\em twisted complexes} $\mathrm{Tw}\,\W(C)$ (see e.g.\ section 3 of
\cite{seidel-book}). The derived wrapped Fukaya category, appearing
in the statement of Theorem \ref{th:main}, is then
defined to be the homotopy category $D\W(C)=H^0(\mathrm{Tw}\,\W(C))$; this
is an honest triangulated category.
By definition, we say that $\W(C)$ is strictly
generated by the objects $L_1,\dots,L_n$ \/if these objects strictly
generate the derived category $D\W(C)$; or equivalently, if
every object of\/ $\W(C)$ is quasi-isomorphic in $\mathrm{Tw}\,\W(C)$ to a
twisted complex built out of the objects $L_1,\dots,L_n$ and their shifts.}

{\rm (3) For the examples we consider in this paper, it turns out
that the difference between strict generation and split-generation is
not important. Indeed, in Appendix B we show that the triangulated
categories $D\W(C)$ and $D_{sg}(W^{-1}(0))$ are actually idempotent complete.
}
\end{remark}

In order to construct the wrapped Fukaya category $\W(C),$
we equip $C$ with a Liouville structure, i.e.\ a 1-form $\lambda$ whose
differential is a symplectic form $d\lambda=\omega,$ and whose associated
Liouville vector field $Z$ (defined by $i_Z\omega=\lambda$) is outward
pointing near the punctures; thus $(C,\lambda)$ has $n$ cylindrical ends modelled
on $(S^1\times [1,\infty),r\,d\theta).$
The objects of $\W(C)$ are (graded) exact
Lagrangian submanifolds of $C$ which are invariant under the Liouville
flow (i.e., radial) inside each cylindrical end (see~\cite{AS,generate} for
details; we will use the same setup as in \cite{generate}).   As a consequence of Theorem \ref{th:fukaya},
the wrapped Fukaya category is independent of the choice of $\lambda$; this can be \emph{a priori} verified using
the fact that, up to adding the differential of a compactly supported function,  any two Liouville structures can be intertwined by a symplectomorphism.

We specifically consider
$n$ disjoint oriented properly embedded arcs $L_1,\dots,L_n\subset C,$ where $L_i$
runs from the $i^\mathrm{th}$ to the $i+1^\mathrm{st}$ cylindrical end
of $C$ (counting mod $n$ as usual),
as shown in Figure \ref{fig:nlags}.
To simplify some aspects of the discussion below,
we will assume that $L_1,\dots,L_n$ are
invariant under the Liouville flow everywhere (not just at infinity);
this can be ensured e.g.\ by constructing the Liouville structure starting from
two discs (the front and back of Figure \ref{fig:nlags}) and attaching $n$
handles whose co-cores are the $L_i.$

Recall that the wrapped Floer complex
$CW^*(L_i,L_j)$ is generated by time 1 chords of the flow $\phi_H^t$ generated by a
Hamiltonian $H:C\to\R$ which is quadratic at infinity (i.e.,
$H(r,\theta)=r^2$ in the cylindrical ends), or equivalently by (transverse)
intersection points of $\phi^1_H(L_i)\cap L_j.$ Without loss of generality we
can assume that, for each $1\le i\le n,$ $H_{|L_i}$ is a Morse function
with a unique minimum.

\begin{figure}[t]
\setlength{\unitlength}{6.5mm}
\begin{picture}(8.6,8)(-4.3,-3.6)
\psset{unit=\unitlength}
\rput{0}{\psellipse(0,4)(1,0.3)\pscurve(1,4)(1.2,2)(1.53,1.76)(3.49,2.19)
  \psline[linewidth=1.5pt]{->}(1.250,1.704)(1.218,1.716)%
  \pscurve[linewidth=1.5pt](0.8,3.8)(1,2)(1.53,1.6)(3.45,2)}
\rput{72}{\psellipse(0,4)(1,0.3)\pscurve(1,4)(1.2,2)(1.53,1.76)(3.49,2.19)
  \psline[linewidth=1.5pt]{->}(1.250,1.704)(1.218,1.716)%
  \pscurve[linewidth=1.5pt](0.8,3.8)(1,2)(1.53,1.6)(3.45,2)}
\rput{144}{\psellipse(0,4)(1,0.3)\pscurve(1,4)(1.2,2)(1.53,1.76)(3.49,2.19)
  \psline[linewidth=1.5pt]{->}(1.250,1.704)(1.218,1.716)%
  \pscurve[linewidth=1.5pt](0.8,3.8)(1,2)(1.53,1.6)(3.45,2)}
\rput{-144}{\psellipse(0,4)(1,0.3)\pscurve(1,4)(1.2,2)(1.53,1.76)(3.49,2.19)
  \psline[linewidth=1.5pt]{->}(1.250,1.704)(1.218,1.716)%
  \pscurve[linewidth=1.5pt](0.8,3.8)(1,2)(1.53,1.6)(3.45,2)}
\rput{-72}{\psellipse(0,4)(1,0.3)\pscurve(1,4)(1.2,2)(1.53,1.76)(3.49,2.19)
  \psline[linewidth=1.5pt]{->}(1.250,1.704)(1.218,1.716)%
  \pscurve[linewidth=1.5pt](0.8,3.8)(1,2)(1.53,1.6)(3.45,2)}
\put(0,4){\makebox(0,0)[cc]{\small 2}}
\put(3.8,1.24){\makebox(0,0)[cc]{\small 1}}
\put(-3.8,1.24){\makebox(0,0)[cc]{\small 3}}
\put(-2.35,-3.24){\makebox(0,0)[cc]{\small \dots}}
\put(2.35,-3.24){\makebox(0,0)[cc]{\small $n$}}
\put(-1.3,1.25){\makebox(0,0)[cc]{\small $L_2$}}
\put(1.3,1.25){\makebox(0,0)[cc]{\small $L_1$}}
\put(1.45,-0.7){\makebox(0,0)[cc]{\small $L_n$}}
\put(-1.45,-0.7){\makebox(0,0)[cc]{\small \dots}}
\psellipticarcn{->}(0,4)(1.2,0.5){-55}{-125}
\psellipticarcn{->}(0,4)(1.3,0.5){-155}{-40}
\put(0,3.35){\makebox(0,0)[ct]{\tiny $u_{1,2}$}}
\put(1.35,4){\makebox(0,0)[lc]{\tiny $v_{2,1}$}}
\rput{-72}{\psellipticarcn{->}(0,4)(1.3,0.5){-100}{-99}}
\put(2.65,1){\tiny $x_1$}
\end{picture}
\caption{The generators of $\W(C)$}
\label{fig:nlags}
\end{figure}
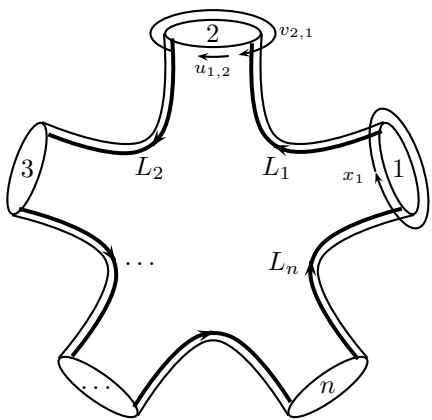

\begin{lemma}\label{l:gradedCW}
The Floer complex $CW^*(L_i,L_j)$ is naturally isomorphic to the vector
space $A(X_i,X_j)$ defined by \eqref{eq:defA}. Moreover, for every choice of
$\Z\text{-}$grading satisfying \eqref{cond} there exists a choice of graded lifts
of $L_1,\dots,L_n$ such that the isomorphism preserves gradings.
\end{lemma}

\begin{figure}[t]
\setlength{\unitlength}{8mm}
\begin{picture}(16,3.1)(-8,-1.6)
\psset{unit=\unitlength}
\psline(-7.5,1)(7.5,1)
\psline[linestyle=dotted](-8,1)(-7,1)
\psline[linestyle=dotted](-8,-1)(-7,-1)
\psline[linestyle=dotted](8,1)(7,1)
\psline[linestyle=dotted](8,-1)(7,-1)
\put(8.2,0){\makebox(0,0)[cl]{\small $i$}}
\put(-8,0){\makebox(0,0)[cr]{\small $i+1$}}
\psline(-7.5,-1)(-3,-1) \pscurve(-3,-1)(-2,-1.1)(-1.5,-1.5)
\psline(7.5,-1)(3,-1) \pscurve(3,-1)(2,-1.1)(1.5,-1.5)
\psline[linestyle=dotted](-1.5,-1.5)(-1.2,-1.9)
\psline[linestyle=dotted](1.5,-1.5)(1.2,-1.9)
\pscurve[linestyle=dashed](0,1)(0.4,0.9)(1,0)(1.6,-0.95)(2,-1.05)
\pscurve(3.5,1)(3.25,0.9)(2.75,0)(2.25,-0.95)(2,-1.05)
\pscurve[linestyle=dashed](3.5,1)(3.75,0.9)(4.25,0)(4.75,-0.9)(5,-1)
\pscurve(6.5,1)(6.25,0.9)(5.75,0)(5.25,-0.9)(5,-1)
\pscurve[linestyle=dashed](6.5,1)(6.75,0.9)(7.25,0)
\pscurve[linestyle=dotted](7.25,0)(7.75,-0.9)(8,-1)
\pscurve(0,1)(-0.4,0.9)(-1,0)(-1.6,-0.95)(-2,-1.05)
\pscurve[linestyle=dashed](-3.5,1)(-3.25,0.9)(-2.75,0)(-2.25,-0.95)(-2,-1.05)
\pscurve(-3.5,1)(-3.75,0.9)(-4.25,0)(-4.75,-0.9)(-5,-1)
\pscurve[linestyle=dashed](-6.5,1)(-6.25,0.9)(-5.75,0)(-5.25,-0.9)(-5,-1)
\pscurve(-6.5,1)(-6.75,0.9)(-7.25,0)
\pscurve[linestyle=dotted](-7.25,0)(-7.75,-0.9)(-8,-1)
\psline[linewidth=1pt]{-<}(-7.25,0.85)(1.75,0.85)
\psline[linewidth=1pt](1.75,0.85)(7.25,0.85)
\psline[linestyle=dotted](-8,0.85)(-7.25,0.85)
\psline[linestyle=dotted](8,0.85)(7.25,0.85)
\psline[linestyle=dotted](8,-0.85)(7.25,-0.85)
\psline[linestyle=dotted](-8,-0.85)(-7.25,-0.85)
\psline[linewidth=1pt]{->}(-7.25,-0.85)(-3,-0.85)
\psline[linewidth=1pt]{-<}(7.25,-0.85)(3,-0.85)
\pscurve[linewidth=1pt](-3,-0.85)(-1.8,-0.95)(-1.3,-1.4)
\psline[linestyle=dotted](-1.3,-1.4)(-1,-1.8)
\pscurve[linewidth=1pt](3,-0.85)(1.8,-0.95)(1.3,-1.4)
\psline[linestyle=dotted](1.3,-1.4)(1,-1.8)
\put(-0.9,-0.2){\small $\phi^1_H(L_i)$}
\put(1.5,1.15){\small $L_i$}
\put(3,-1.4){\small $L_{i-1}$}
\put(-3.65,-1.4){\small $L_{i+1}$}
\pscircle*(-6.8,0.85){0.08}
\pscircle*(-3.8,0.85){0.08}
\pscircle*(6.2,0.85){0.08}
\pscircle*(3.2,0.85){0.08}
\pscircle*(-0.45,0.85){0.08}
\put(-6.5,0.75){\makebox(0,0)[ct]{\tiny $y_{i}^2$}}
\put(-3.5,0.65){\makebox(0,0)[ct]{\tiny $y_{i}$}}
\put(-0.15,0.7){\makebox(0,0)[ct]{\tiny $\mathrm{id}_{L_i}$}}
\put(6.4,0.75){\makebox(0,0)[ct]{\tiny $x_{i}^2$}}
\put(3.4,0.65){\makebox(0,0)[ct]{\tiny $x_{i}$}}
\pscircle*(-4.7,-0.85){0.08}
\pscircle*(-1.7,-1){0.08}
\pscircle*(5.3,-0.85){0.08}
\pscircle*(2.35,-0.85){0.08}
\put(-4.7,-1.1){\makebox(0,0)[tc]{\tiny $u_{i,i+1}y_i$}}
\put(-1.5,-1){\makebox(0,0)[lc]{\tiny $u_{i,i+1}$}}
\put(2.5,-0.7){\makebox(0,0)[rb]{\tiny $v_{i,i-1}$}}
\put(5.5,-1.1){\makebox(0,0)[tc]{\tiny $v_{i,i-1}x_i$}}
\end{picture}
\caption{Generators of the wrapped Floer complexes}
\label{fig:CW}
\end{figure}
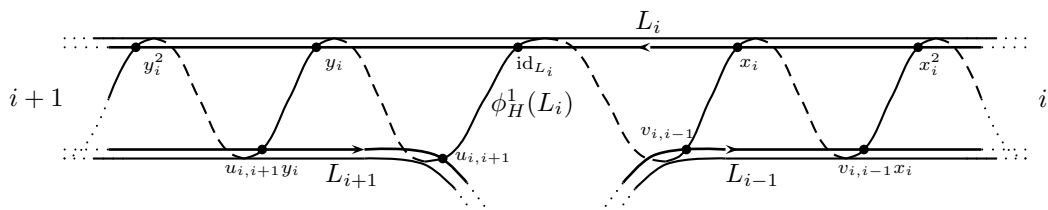

\proof
The intersections between $\phi^1_H(L_i)$ and $L_i$ (resp.\ $L_{i\pm 1}$) are
pictured in Figure \ref{fig:CW}. The point of $\phi^1_H(L_i)\cap L_i$ which corresponds to the
minimum of $H_{|L_i}$ is labeled by the identity element, while the
successive intersections in the $i^\mathrm{th}$ end are labeled by powers
of $x_i,$ and similarly those in the $(i+1)^\mathrm{st}$ end are labeled by
powers of $y_{i}.$ The generators of $CW^*(L_i,L_{i+1})$ (i.e.,
points of $\phi^1_H(L_i)\cap L_{i+1}$) are labeled by $u_{i,i+1} y_i^k,$
$k=0,1,\dots,$ and similarly the generators of $CW^*(L_i,L_{i-1})$
are labeled by $v_{i,i-1} x_i^k$ (see Figure \ref{fig:CW}).

Recall that a $\Z\text{-}$grading on Floer complexes requires the choice of a
trivialization of $TC.$ Denote by $d_i\in\Z$ the rotation number of a simple
closed curve encircling the $i^\mathrm{th}$ puncture of $C$ with respect
to the chosen trivialization: by an Euler characteristic argument,
$\sum d_i=n-2.$ Observing that each rotation
around the $i^\mathrm{th}$ cylindrical end contributes $2d_i$ to the
Maslov index, we obtain that $\deg(x_i^k)=2kd_i,$ and similarly
$\deg(y_{i}^{k})=2kd_{i+1}.$

The freedom to choose graded lifts of the Lagrangians $L_i$ (compatibly
with the given orientations) means that $p_i=\deg(u_{i-1,i})$ can be any
odd integer for $i=2,\dots,n$; however, considering the $n$-gon obtained
by deforming the front half of Figure \ref{fig:nlags}, we obtain the
relation $p_1+\dots+p_n=n-2.$ Moreover, comparing the Maslov indices of
the various morphisms between $L_{i-1}$ and $L_i$ in the $i^\mathrm{th}$
end we obtain that $\deg(x_i^ku_{i-1,i})=p_i+2kd_i,$
$\deg(v_{i,i-1})=2d_i-p_i,$ and $\deg( v_{i,i-1}x_i^l)=2d_i-p_i+2ld_i.$
Setting $q_i=2d_i-p_i,$ this completes the proof.
\endproof

It follows immediately from Lemma \ref{l:gradedCW} that the Floer
differential on $CW^*(L_i,L_j)$ is identically zero, since
the degrees of the generators all have the same parity.

\begin{lemma}\label{l:HW}
There is a natural isomorphism of algebras
$$
\bigoplus_{i,j} \,HW^*(L_i,L_j)\simeq \bigoplus_{i,j} A(X_i,X_j)
$$
where $A$ is the $\kk$-linear category defined by \eqref{eq:defA}.
\end{lemma}

\proof
Recall from \cite[Section 3.2]{generate} that the product on wrapped
Floer cohomology can be defined by counting
solutions to a perturbed Cauchy-Riemann equation. Namely,
one considers finite energy maps $u:S\to C$
satisfying an equation of the form
\begin{equation}\label{eq:pertJhol}
(du-X_H\otimes\alpha)^{0,1}=0.
\end{equation}
Here the domain $S$ is a disc with three
strip-like ends, and $u$ is required to map $\partial S$ to the images
of the respective Lagrangians under suitable Liouville rescalings
(in our case $L_i$ is invariant under the Liouville
flow, so $\partial S$ is mapped to $L_i$);
$X_H$ is the Hamiltonian vector field generated by $H,$
and $\alpha$ is a closed 1-form on $S$ such that $\alpha_{|\partial S}=0$
and which is standard in the strip-like ends
(modelled on $dt$ for the input ends, $2\,dt$ for the output end).
(Further perturbations of $H$ and $J$ would be required to achieve
transversality in general, but are not necessary in our case.)

The equation (\ref{eq:pertJhol}) can be rewritten as a standard
holomorphic curve equation (with a domain-dependent almost-complex structure)
by considering $$\tilde{u}=\phi_H^\tau\circ u:S\to C,$$ where $\tau:S\to
[0,2]$ is a primitive of $\alpha.$ The product on
$CW^*(L_j,L_k)\otimes CW^*(L_i,L_j)$ is then the
usual Floer product
$$CF^*(\phi_H^1(L_j),L_k)\otimes CF^*(\phi_H^2(L_i),\phi_H^1(L_j))\to
CF^*(\phi_H^2(L_i),L_k),$$
where the right-hand side is identified with
$CW^*(L_i,L_k)$ by a rescaling trick \cite{generate}.

With this understood, since we are interested in rigid holomorphic discs,
the computation of the product structure is
simply a matter of identifying all immersed polygonal regions in $C$ with
boundaries on $\phi_H^2(L_i),$ $\phi_H^1(L_j)$ and $L_k$ and satisfying
a local convexity condition at the corners. (Simultaneous compatibility of
the product structure with all $\Z\text{-}$gradings satisfying \eqref{cond}
drastically reduces the number of cases to consider.) Signs are determined
as in \cite[Section 13]{seidel-book}, and in our case they all turn out to
be positive for parity reasons.

As an example, Figure \ref{fig:product} shows the triangle which yields
the identity $u_{i-1,i}\circ v_{i,i-1}=x_i.$ (The triangle corresponding
to $u_{i-1,i}\circ (v_{i,i-1} x_i)=x_i^2$ is also visible.)
\endproof

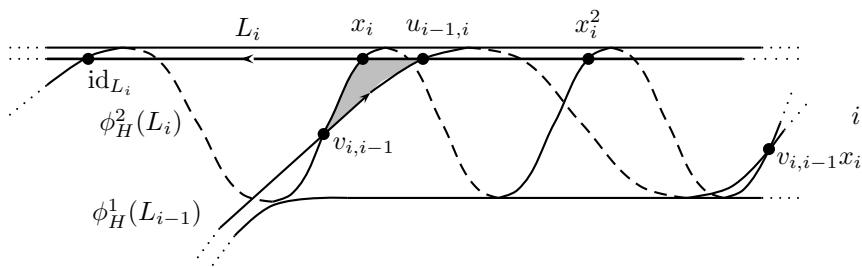
\begin{figure}[t]
\setlength{\unitlength}{1cm}
\begin{picture}(11,3.1)(-1.5,-1.6)
\psset{unit=\unitlength}
\psline(-1,1)(8.5,1)
\psline[linestyle=dotted](-1.5,1)(-1,1)
\psline[linestyle=dotted](-1.5,0.85)(-1,0.85)
\psline[linestyle=dotted](9,1)(8.5,1)
\psline[linestyle=dotted](9,-1)(8.5,-1)
\put(9.7,0.1){\makebox(0,0)[cl]{\small $i$}}
\psline(8.5,-1)(3,-1) \pscurve(3,-1)(2,-1.1)(1.5,-1.5)
\psline[linestyle=dotted](1.5,-1.5)(1.2,-1.9)
\pscurve[linestyle=dashed](3.5,1)(3.75,0.9)(4.25,0)(4.75,-0.9)(5,-1)
\pspolygon[linestyle=none,fillstyle=solid,fillcolor=lightgray]%
  (2.68,-0.15)(3.45,0.55)(4,0.85)(3.2,0.85)(3,0.6)
\pscurve[linestyle=dashed](0,1)(0.4,0.9)(1,0)(1.6,-0.95)(2,-1.05)
\pscurve(3.5,1)(3.25,0.9)(2.75,0)(2.25,-0.95)(2,-1.05)
\pscurve(6.5,1)(6.25,0.9)(5.75,0)(5.25,-0.9)(5,-1)
\pscurve[linestyle=dashed](6.5,1)(6.75,0.9)(7.25,0)(7.75,-0.9)(8,-1)
\pscurve(8,-1)(8.25,-0.9)(8.75,0)
\psline[linestyle=dotted](8.75,0)(8.9,0.4)
\pscurve(0,1)(-0.4,0.9)(-1,0.5)
\psline[linestyle=dotted](-1,0.5)(-1.5,0.1)
\psline[linewidth=1pt]{-<}(-1,0.85)(1.75,0.85)
\psline[linewidth=1pt](1.75,0.85)(8.25,0.85)
\psline[linestyle=dotted](9,0.85)(8.25,0.85)
\pscurve(3.3,0.4)(4,0.85)(4.6,1)
\pscurve[linestyle=dashed](4.6,1)(5.3,0.85)(6.1,0)(6.9,-0.85)(7.5,-1)
\pscurve(7.5,-1)(8.1,-0.85)(8.8,-0.1)
\psline[linestyle=dotted](8.8,-0.1)(9.1,0.25)
\pscurve{<-}(3.3,0.4)(1.8,-0.95)(1.3,-1.4)
\psline[linestyle=dotted](1.3,-1.4)(1,-1.8)
\put(-0.3,-0.1){\small $\phi^2_H(L_i)$}
\put(1.5,1.15){\small $L_i$}
\put(-0.4,-1.3){\small $\phi_H^1(L_{i-1})$}
\pscircle*(6.2,0.85){0.08}
\pscircle*(3.2,0.85){0.08}
\pscircle*(-0.45,0.85){0.08}
\put(-0.15,0.7){\makebox(0,0)[ct]{\small $\mathrm{id}_{L_i}$}}
\put(6.2,1.12){\makebox(0,0)[cb]{\small $x_{i}^2$}}
\put(3.2,1.17){\makebox(0,0)[cb]{\small $x_{i}$}}
\pscircle*(2.68,-0.15){0.08}
\pscircle*(4,0.85){0.08}
\pscircle*(8.6,-0.35){0.08}
\put(2.8,-0.2){\makebox(0,0)[lt]{\small $v_{i,i-1}$}}
\put(8.7,-0.4){\makebox(0,0)[lt]{\small $v_{i,i-1}x_i$}}
\put(4.2,1.12){\makebox(0,0)[cb]{\small $u_{i-1,i}$}}
\end{picture}
\caption{A holomorphic triangle contributing to the product}
\label{fig:product}
\end{figure}

\begin{lemma}\label{l:mn}
In $\W(C)$ we have
$$
m_n(u_{i-1,i},u_{i-2,i-1},\dots,u_{i,i+1})=\mathrm{id}_{L_i}\ \ \mathrm{and}\
\  m_n(v_{i+1,i},v_{i+2,i+1},\dots,v_{i,i-1})=(-1)^n\mathrm{id}_{L_i}.
$$
\end{lemma}

\proof
Since $m_n(u_{i-1,i},\dots,u_{i,i+1})$ has degree $0$ for all
gradings satisfying \eqref{cond}, it must be a scalar multiple of
$\mathrm{id}_{L_i}.$ By the same argument as in Lemma \ref{l:HW}, the
calculation reduces to an enumeration of immersed $(n+1)$-sided
polygonal regions
with boundary on $\phi_H^n(L_i),$ $\phi_H^{n-1}(L_{i+1}),$ \dots,
$\phi_H^1(L_{i-1}),$ and $L_i,$ with locally convex corners at the prescribed intersection
points. Recall that $u_{j,j+1}$ is the first intersection point between
the images of $L_j$ and $L_{j+1}$ created by the wrapping flow inside the
$(j+1)^\mathrm{st}$ cylindrical end, and can also be visualized as a chord from
$L_j$ to $L_{j+1}$ as pictured in Figure \ref{fig:nlags}. The only polygonal
region which contributes to $m_n$ is therefore the front half of Figure
\ref{fig:nlags} (deformed by the wrapping flow).  Since the orientation of
the boundary of the polygon agrees with that of the $L_j$'s,
its contribution to the coefficient of $\mathrm{id}_{L_i}$ in
$m_n(u_{i-1,i},u_{i-2,i-1},\dots,u_{i,i+1})$ is $+1$
(cf.\ \cite[\S 13]{seidel-book}).

The argument is the same for $m_n(v_{i+1,i},\dots,v_{i,i-1}),$ except the
polygon which contributes now corresponds to the back half of Figure
\ref{fig:nlags}. Since the orientation of the boundary of the polygon
differs from that of the $L_j$'s, and $\deg(v_{j,j-1})=q_j$ is odd for all
$j=1,\dots,n,$ the coefficient of $\mathrm{id}_{L_i}$ is now $(-1)^n.$
\endproof

By Lemma \ref{classification}, we conclude that
the $A_\infty\text{-}$structure on $\bigoplus_{i,j} \Hom(L_i,L_j)$
is strictly homotopic to $m^{1,(-1)^n}.$ The sign discrepancy can be
corrected by changing the identification between the two categories:
namely, the automorphism of $\widetilde{A}$ which maps $u_{i,i+1}$ to
itself, $v_{i,i-1}$ to $-v_{i,i-1},$ and $x_i$ to $-x_i$ intertwines
the $A_\infty\text{-}$structures $m^{1,(-1)^n}$ and $m^{1,1}.$

The final ingredient needed for Theorem \ref{th:fukaya} is the following
generation statement:

\begin{lemma}\label{l:generate}
$\W(C)$ is strictly generated by $L_1,\dots,L_{n-1}.$
\end{lemma}

\proof
Observe that $C$ can be viewed as an $n$-fold simple branched covering
of $\C$ with $2n-2$ branch points, around which the
monodromies are successively $(1\ 2),$ $(2\ 3),$ \dots, $(n-1\ n),$
$(n-1\ n),$ \dots, $(2\ 3),$ $(1\ 2)$; see Figure \ref{fig:cover}.
(Since the product of these
transpositions is the identity, the monodromy at infinity is trivial, and
it is easy to check that the $n$-fold cover we have described is indeed
an $n$-punctured $\PP^1$).

\begin{figure}[t]
\setlength{\unitlength}{4mm}
\begin{picture}(10,9.8)(-5,-0.5)
\psset{unit=\unitlength}
\psframe(-5,-0.5)(5,9)
\pscircle*(0,8){0.1}
\pscircle*(0,7){0.1}
\pscircle*(0,5){0.1}
\pscircle*(0,4){0.1}
\pscircle*(0,2){0.1}
\pscircle*(0,1){0.1}
\psline[linestyle=dotted](0,6.3)(0,5.7)
\psline[linestyle=dotted](0,3.3)(0,2.7)
\psline[linestyle=dotted](5.6,6.3)(5.6,5.7)
\psline[linestyle=dotted](5.6,3.3)(5.6,2.7)
\psline(0,8)(5,8)
\psline(0,7)(5,7)
\psline(0,5)(5,5)
\psline(0,4)(5,4)
\psline(0,2)(5,2)
\psline(0,1)(5,1)
\put(-0.25,8){\makebox(0,0)[rc]{\tiny $(1\ 2)$}}
\put(-0.25,7){\makebox(0,0)[rc]{\tiny $(2\ 3)$}}
\put(-0.25,5){\makebox(0,0)[rc]{\tiny $(n\!-\!1\ n)$}}
\put(-0.25,4.2){\makebox(0,0)[rc]{\tiny $(n\!-\!1\ n)$}}
\put(-0.25,2){\makebox(0,0)[rc]{\tiny $(2\ 3)$}}
\put(-0.25,1){\makebox(0,0)[rc]{\tiny $(1\ 2)$}}
\put(5.2,1){\makebox(0,0)[lc]{\small $\delta_1$}}
\put(5.2,2){\makebox(0,0)[lc]{\small $\delta_2$}}
\put(5.2,4){\makebox(0,0)[lc]{\small $\delta_{n-1}$}}
\put(5.2,5){\makebox(0,0)[lc]{\small $\delta_n$}}
\put(5.2,7){\makebox(0,0)[lc]{\small $\delta_{2n-3}$}}
\put(5.2,8){\makebox(0,0)[lc]{\small $\delta_{2n-2}$}}
\pscurve[linestyle=dashed,dash=1pt 2pt](0,4)(-1.5,3)(-2.2,2)(-2,0.5)(-1,0.1)(0,0)
\psline[linestyle=dashed,dash=1pt 2pt](0,0)(5,0)
\put(-2.8,2.5){\makebox(0,0)[lc]{\small $\varepsilon$}}
\end{picture}
\caption{A simple branched cover $\pi:C\to\C$}
\label{fig:cover}
\end{figure}
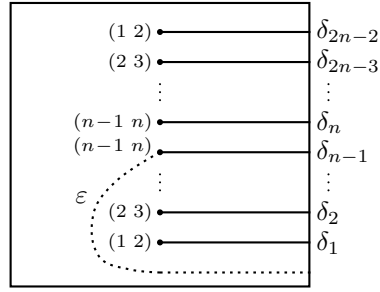

The $2n-2$ thimbles $\delta_1,\dots,\delta_{2n-2}$ are disjoint properly
embedded arcs in $C,$ projecting to the
arcs shown in Figure \ref{fig:cover}. We claim that
they are respectively isotopic to $L_1, \dots, L_{n-1},$
$L_{n-1},
 \dots, L_1$ in that order. Indeed, for $1\le i\le n-1,$ $\delta_i$ and
$\delta_{2n-1-i}$ both connect the $i^\mathrm{th}$ and $(i+1)^\mathrm{st}$
punctures of $C.$ Cutting $C$ open along all these arcs, we obtain
$n$ components, one of them (corresponding to the first sheet of the
covering near $-\infty$) a $(2n-2)$-gon bounded successively by
$\delta_1,\delta_2,\dots,\delta_{2n-2},$ while the $n-1$ others
(corresponding to sheets $2,\dots,n$ near $-\infty$) are strips
bounded by $\delta_i$ and $\delta_{2n-1-i}.$ From there it is not hard
to check that $\delta_i$ and $\delta_{2n-1-i}$ are both isotopic to $L_i$
for $1\le i\le n-1.$

The result then follows from Theorem \ref{thm:thimbles_generate}, which
asserts that the thimbles $\delta_1,\dots,\delta_{2n-2}$ strictly generate $\W(C).$
\endproof

Note that, by this result, $L_n$ could have been omitted entirely from the
discussion. 
To be more specific, an argument similar to that in Appendix A shows that,
up to a shift, $L_n$ is quasi-isomorphic to the complex
$$L_1\stackrel{u_{1,2}}{\longrightarrow}
L_2\stackrel{u_{2,3}}{\longrightarrow}\ \dots
\stackrel{u_{n-2,n-1}}{\longrightarrow} L_{n-1}.
$$
(Namely, consider a double branched 
cover as in Appendix A, and denote by $\gamma_i$ the curve obtained by
doubling the thimble $\delta_i$. 
The thimble $\varepsilon$ corresponding to the dotted
arc in Figure \ref{fig:cover} is isotopic to $L_n$. 
However, by Proposition 18.23 of \cite{seidel-book}, 
the curve obtained by doubling $\varepsilon$ is
isotopic to the image of $\gamma_{n-1}$ under the product of the Dehn
twists about $\gamma_{n-2},\dots,\gamma_1$, and can be interpreted as
an iterated mapping cone; the claim then follows from the same argument
as in the proof of Theorem \ref{thm:thimbles_generate}.)

We shall encounter this complex on the mirror side (see Equation \eqref{complex})  in the process of determining the
$A_{\infty}$ structure on the category of matrix factorizations.  In particular, we could replace Lemma \ref{l:mn}
with an argument modeled after that given for Lemma \ref{l:mndsg}.

\section{The Landau-Ginzburg mirror $(X(n), W)$}
\label{s:LGmodel}

In this section we describe mirror Landau-Ginzburg (LG) models
$W:X(n)\to\C$ for $n\ge 3.$ These mirrors are toric,
and their construction can be justified by a physics
argument due to Hori and Vafa \cite{HV}, see also
\cite[Section 3]{KKOY}.
(Mathematically, this construction can be construed as a duality between
toric Landau-Ginzburg models.)

Let us start with $\PP^1$ minus three points. In this case we can
realize our curve as a line in $(\C^{*})^{2}$ viewed as the
complement of three lines in $\PP^2.$
The Hori--Vafa procedure then gives us as mirror LG model a variety
$X(3)\subset \C^4$ defined by the equation
$$
\mathrm{x_1 x_2 x_3}=\exp(-t) \mathrm{p}
$$
with superpotential $W={\rm p}: X(3)\to \C,$ i.e. the mirror LG model $(X(3), W)$
is isomorphic to the affine space $\C^3$ with the superpotential
$W=\mathrm{x_1 x_2 x_3}.$

In the case $n=2k$ we can realize $C=\PP^1\backslash\{2k\ \text{points}\}$ as a curve  of bidegree
$(k-1, 1)$ in the torus $(\C^{*})^2$ considered as the open orbit of $\PP^1\times\PP^1.$
The raw output of the Hori--Vafa procedure is
a singular  variety $Y(2k)\subset \C^5$ defined by the equations
$$
\begin{cases}
\mathrm{y_1 \cdot y_4= y_3^{\mathit{k}-1}}\\
\mathrm{y_2 \cdot y_5= y_3}
\end{cases}
$$
with $y_3$ as a superpotential. The variety $Y(2k)$ is a 3-dimensional affine toric variety with coordinate algebra
$\C [ \mathrm{y_1, y_2, y_3 y_2^{-1}, y_3^{\mathit{k}-1} y_1^{-1}} ].$
A smooth mirror $(X(2k), W)$ can then be obtained by resolving the
singularities of $Y(2k).$
More precisely, $Y(2k)$ admits toric small resolutions. Any two
such resolutions are related to each other by flops, and thus yield LG
models which are
equivalent, in the sense that they have equivalent  categories of D-branes of type B (see \cite{KKOY}).

If $n$ is odd we realize our curve as a curve in the Hirzebruch surface
$\mathbb{F}_1$.
All the calculations are similar.

Now we describe a mirror LG model $(X(n), W)$ directly.
Consider the lattice $N=\Z^3$ and the
fan $\Sigma_n$ in $N$ with the following maximal cones:
\begin{alignat*}{2}
& \sigma_{i,0}:=\left\langle (i,0,1),(i,1,1),(i+1,0,1)\right\rangle, &\quad 0\leq i < \left\lfloor
\frac{n-1}{2}\right\rfloor,\\
& \sigma_{i,1}:=\left\langle (i,1,1),(i+1,1,1),(i+1,0,1)\right\rangle, &\quad 0\leq i < \left\lfloor
\frac{n-2}{2}\right\rfloor.
\end{alignat*}

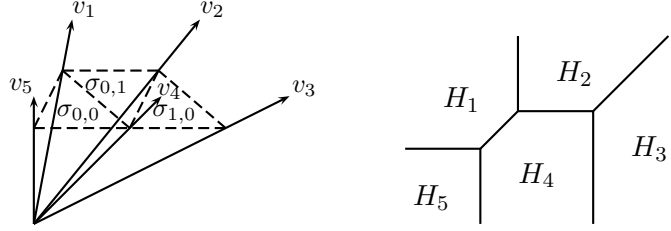
\begin{figure}[t]
\setlength{\unitlength}{17mm}
\begin{picture}(2.3,1.9)(0,0)
\psset{unit=\unitlength,dash=4pt 2pt}
\psline{->}(0,0)(0,1)  
\psline{->}(0,0)(0.3,1.6)  
\psline{->}(0,0)(1,1)  
\psline{->}(0,0)(1.3,1.6)  
\psline{->}(0,0)(2,1)  
\psline[linestyle=dashed](0,0.75)(0.225,1.2)
\psline[linestyle=dashed](0.225,1.2)(0.75,0.75)
\psline[linestyle=dashed](0,0.75)(1.5,0.75)
\psline[linestyle=dashed](0.225,1.2)(0.975,1.2)
\psline[linestyle=dashed](0.75,0.75)(0.975,1.2)
\psline[linestyle=dashed](1.5,0.75)(0.975,1.2)
\put(0.18,0.84){\small $\sigma_{0,0}$}
\put(0.93,0.84){\small $\sigma_{1,0}$}
\put(0.4,1.05){\small $\sigma_{0,1}$}
\put(0,1.02){\makebox(0,0)[rb]{\small $v_5$}}
\put(1.06,0.98){\makebox(0,0)[cb]{\small $v_4$}}
\put(2.02,1){\makebox(0,0)[lb]{\small $v_3$}}
\put(0.3,1.62){\makebox(0,0)[lb]{\small $v_1$}}
\put(1.3,1.62){\makebox(0,0)[lb]{\small $v_2$}}
\end{picture}
\qquad
\setlength{\unitlength}{1cm}
\begin{picture}(3.5,2.5)(0,0)
\psset{unit=\unitlength}
\psline(0,1)(1,1)(1,0)
\psline(1,1)(1.5,1.5)
\psline(1.5,2.5)(1.5,1.5)(2.5,1.5)(2.5,0)
\psline(2.5,1.5)(3.5,2.5)
\put(0.35,0.35){\makebox(0,0)[cc]{$H_5$}}
\put(1.75,0.65){\makebox(0,0)[cc]{$H_4$}}
\put(3.25,1){\makebox(0,0)[cc]{$H_3$}}
\put(2.25,2){\makebox(0,0)[cc]{$H_2$}}
\put(0.75,1.65){\makebox(0,0)[cc]{$H_1$}}
\end{picture}
\caption{The fan $\Sigma$ and the configuration of divisors $H_i$ (for $n=5$)}
\label{fig:LG}
\end{figure}

\noindent
Let $X(n):=X_{\Sigma_n}$ be the toric variety corresponding to the fan $\Sigma_n.$

We label the one-dimensional cones in $\Sigma_n$ as follows:
$$
{\mathrm v}_i:=(i-1,1,1),\quad 1\leq i\leq \left\lfloor\frac{n}{2}\right\rfloor,\quad {\mathrm v}_i=(n-i,0,1),
\quad \left\lfloor\frac{n}{2}\right\rfloor+1\leq i\leq n.
$$
For simplicity, we set ${\mathrm v}_{i-n}:={\mathrm v}_i=:{\mathrm v}_{i+n}.$ Also, let
$H_i:=H_{\mathrm v_i}\subset X(n)$ be the toric divisor
corresponding to the ray ${\mathrm v}_i$ (see Figure \ref{fig:LG}).

The vector $\xi=(0,0,1)\in M=N^{\vee}$ is non-negative on each cone of
$\Sigma_n,$ and therefore it defines a function
$$
W=W_{\xi}:X(n)\to \C,
$$
which will be considered as the superpotential. By construction,
$W^{-1}(0)=\bigcup_{i=1}^n H_i.$

The LG model $(X(n),W)$ can be considered as a mirror to $C=\PP^1\setminus\{n\text{
points}\},$ by the argument explained above.

\begin{remark}
{\rm The construction of the LG model $(X(n),W)$ can also be motivated
from the perspective of the Strominger-Yau-Zaslow conjecture. Here again
we think of $C$ as a curve in a toric surface, namely we
write $C=\overline{C}\cap (\C^*)^2,$ where $\overline{C}$ is a rational
curve
in either $\PP^1\times\PP^1$ (for $n$ even) or the Hirzebruch surface
$\mathbb{F}_1$ (for $n$ odd). Then, by the main result of \cite{AAK},
$(X(n),W)$ is an SYZ mirror to the blowup
of $(\C^*)^2\times\C$ along $C\times \{0\}.$
}
\end{remark}

\section{The category of D-branes of type B in LG model $(X(n), W)$}
\label{s:Dbsing}

The aim of this section is to describe the category of D-branes of type B in the mirror symmetric
LG model $(X(n),W),$ and to show that it is equivalent to the derived category of the wrapped Fukaya category $\W(C)$
calculated in Section \ref{s:wrapped}.

There are two ways to define the category of D-branes of type B in  LG models.
Assuming that $W$ has a unique critical value at the origin, the first one is to take the triangulated category of singularities $D_{sg}(X_0)$ of
the singular fiber $X_0=W^{-1}(0),$ which is by definition the Verdier quotient of the bounded derived category
of coherent sheaves $D^b(\coh(X_0))$ by the full subcategory of perfect complexes
$\perf{X_0}.$

The other approach involves matrix factorizations.
We can define a triangulated category of matrix factorizations $MF(X,W)$  as follows.
First define a category $MF^{naive}(X,W)$ whose objects are
pairs
$$
\underline{T}:=\Bigl(
\xymatrix{
T_1 \ar@<0.6ex>[r]^{t_1} & T_0 \ar@<0.6ex>[l]^{t_0}
}
\Bigl),
$$
 where $T_1, T_0$ are locally free sheaves of finite rank on $X,$
and $t_1$ and  $t_0$ are
morphisms such that both compositions $t_1\cdot t_0$ and $t_0\cdot t_1$ are multiplication by $W.$
Morphisms in the category $MF^{naive}(X,W)$ are morphisms of pairs modulo null-homotopic morphisms,
 where a morphism of pairs $f:\underline{T}\to\underline{S}$  is a pair of morphisms
$f_1: T_1\to S_1$ and $f_0: T_0\to S_0$ such that
$f_1\cdot t_0=s_0\cdot f_0$ and $s_1\cdot f_1=f_0\cdot t_1,$
and a morphism $f$ is null-homotopic if there are two morphisms
$h_0: T_0\to S_1$ and $h_1: T_1\to S_0$ such that
$f_1=s_0 h_1 + h_0 t_1$ and $f_0=h_1 t_0 + s_1 h_0.$

The category $MF^{naive}(X,W)$ can be
endowed with a natural triangulated structure.
Now, we consider the full triangulated subcategory  of acyclic objects,
namely the subcategory
$Ac(X, W)\subset MF^{naive}(X,W)$ which consists of all convolutions of exact triples of matrix factorizations.
We define a triangulated category  of matrix factorizations $MF(X, W)$  on $(X, W)$
as the Verdier quotient of $MF^{naive}(X,W)$ by the subcategory of acyclic objects
$$
MF(X,W):=MF^{naive}(X,W)/Ac(X, W).
$$
This category will  also be called
triangulated category of D-branes of type B in the LG model $(X, W).$
It is proved in  \cite{Orlov2}
that there is an equivalence
\begin{equation}\label{MF_D_sg}
MF(X,W)\stackrel{\sim}{\lto} D_{sg}(X_0),
\end{equation}
where the functor \eqref{MF_D_sg} is defined by the rule
$\underline{T}\mapsto \coker(t_1)$ and  we can regard $\coker(t_1)$ as a sheaf on $X_0$ due to it being annihilated by $W$ as a sheaf on $X.$

In this section we use the first approach and work with the triangulated  category of singularities $D_{sg}(X_0).$
This category has a natural DG enhancement, which arises as the DG quotient of the natural DG enhancement
of $D^b(\coh (X_0))$ by the DG subcategory of perfect complexes $\perf{X_0}.$
This implies that the triangulated category of singularities $D_{sg}(X_0)$ has a natural minimal
$A_{\infty}$-structure which is quasi-equivalent to the DG enhancement described above.
Thus, in the following discussion we will consider the triangulated category of singularities $D_{sg}(X_0)$ with this natural
$A_{\infty}$-structure.

The singular fiber $X_0$ of $W$ is the union of
the toric divisors in
$X(n).$ Consider the structure sheaves $E_i:=\cO_{H_i}$ as objects of the category $D_{sg}(X_0).$

\begin{theo} \label{th:B-side-mn} Let $(X(n), W)$ be the LG model described above.
Then the triangulated category of singularities $D_{sg}(X_0)$ of the singular fiber
$X_0=W^{-1}(0)$ is strictly generated by $n$ objects $E_1,\dots, E_n$ and
there is a natural isomorphism of  algebras
$$
\bigoplus_{\substack{i,j}}\Hom_{D_{sg}(X_0)}(E_i,E_j)\cong \bigoplus_{i,j} A(X_i,X_j),
$$
where $A$ is the category defined in \eqref{eq:defA}.

Moreover,
the $A_{\infty}$-structure on $\bigoplus_{\substack{i,j}}\Hom_{D_{sg}(X_0)}(E_i,E_j)$ is strictly homotopic to $m^{(1,1)}.$
\end{theo}

Each object $E_i=\cO_{H_i},$ being the cokernel of the morphism $\cO_{X(n)}(-H_i)\to \cO_{X(n)},$
is a Cohen-Macaulay sheaf on the fiber $X_0.$ Hence by Proposition 1.21 of \cite{Orlov}
we have
$$
\Hom_{D_{sg}(X_0)}(E_i, E_j[N])\cong\Ext^N_{X_0}(E_i, E_j)
$$
for any $N>\dim X_0=2.$
Since the shift by $[2]$ is isomorphic to identity, this allows us to determine morphisms between these objects in
$D_{sg}(X_0)$ by calculating $\Ext$'s between them in the category of coherent sheaves.
Hence, if $H_i\cap H_j=\emptyset,$ then $\Hom_{D_{sg}(X_0)}^{\bdot}(E_i,E_j)=0.$

Assume that $H_i\cap H_j\ne \emptyset,$ and
denote by $\varGamma_{ij}$ the curve that is the intersection of $H_i$ and $H_j.$
Consider the 2-periodic locally free resolution of $\cO_{H_i}$ on $X_0,$
$$
\{\cdots\lto\cO_{X_0}\lto\cO_{X_0}(-H_i)\lto\cO_{X_0}\}\lto\cO_{H_i}\lto 0.
$$
Now the groups $\Ext^N_{X_0}(E_i, E_j)$ can be calculated as the hypercohomology of
the 2-periodic complex
$$
0\lto \cO_{H_j}\stackrel{\phi_{ij}}{\lto}\cO_{H_j}(H_i)\stackrel{\psi_{ij}}{\lto}\cO_{H_j}\lto\cdots
$$

We first consider the case where $j=i$: then $\phi_{ii}=0,$ and the morphism $\psi_{ii}$ is isomorphic to the canonical map
$\cO_{H_i}(-D_i)\to\cO_{H_i},$ where $D_i=\bigcup _j \varGamma_{ij}.$ Hence the cokernel
of $\psi_{ii}$ is the structure sheaf $\cO_{D_i}.$ This implies that $\Hom_{D_{sg}(X_0)}^{\bdot}(E_i,E_i)$
is concentrated in even degree and
the algebra $\Hom_{D_{sg}(X_0)}^{0}(E_i,E_i)$ is isomorphic to the algebra of regular functions on $D_i.$
However, $D_i$ consists of either two $\bbA^1$ meeting at one point, two $\bbA^1$ connected by a $\bbP^1,$ or
two $\bbA^1$ connected by a chain of two $\bbP^1$ (see Figure \ref{fig:LG}). In all cases,
the algebra of regular functions is isomorphic to $\kk[x_i, y_i]/(x_i y_i).$

On the other hand, when $j\ne i$\/ we must have $\psi_{ij}=0,$ and the cokernel of $\phi_{ij}$ is isomorphic to $\cO_{\varGamma_{ij}}(H_i).$
When $j\not\in \{i,i\pm 1\}$ the curve $\varGamma_{ij}$ is isomorphic to $\PP^1,$ and moreover the normal bundles to $\varGamma_{ij}$ in
$H_i$ and in $H_j$ are both isomorphic to $\cO_{\PP^1}(-1).$ Hence $\cO_{\varGamma_{ij}}(H_i)\cong\cO_{\PP^1}(-1)$
and we obtain that $\Hom_{D_{sg}(X_0)}^{\bdot}(E_i,E_j)$ is trivial.

When $j=i+1,$ the curve $\varGamma_{ij}$ is isomorphic to $\bbA^1$ and $\Hom_{D_{sg}(X_0)}^{\bdot}(E_i,E_j)$
is concentrated in odd degree. Moreover,
$\Hom_{D_{sg}(X_0)}(E_i,E_j[1])$ is isomorphic to $H^0(\cO_{\varGamma_{ij}}).$ Therefore,
it is generated by a morphism $u_{i, i+1}:E_i\to E_{i+1}[1]$ as a right module over $\End(E_i)$ and as a left module over $\End(E_{i+1}),$
and there are isomorphisms
$$
\Hom_{D_{sg}(X_0)}(E_i,E_{i+1}[1])\cong \kk[x_{i+1}] u_{i, i+1}= u_{i, i+1} \kk[y_i].
$$

Analogously, if $j=i-1$ then there is a morphism $v_{i, i-1}: E_i\to E_{i-1}[1]$ such that
$$
\Hom_{D_{sg}(X_0)}(E_i,E_{i-1}[1])\cong \kk[y_{i-1}] v_{i, i-1}= v_{i, i-1} \kk[x_i].
$$
It is easy to check that the composition $v_{i+1, i} u_{i, i+1}$ is equal to $y_i$ and $u_{i, i-1} v_{i, i-1}=x_i.$

Hence, we obtain an isomorphism of super-algebras
$$
\bigoplus_{\substack{i,j}}\Hom_{D_{sg}(X_0)}(E_i,E_j)\cong \bigoplus_{i,j} A(X_i,X_j)
$$
This proves the first part of the Theorem.

We claim that the $\Z/2$-graded algebra
$\bigoplus_{i,j}\Hom_{D_{sg}(X_0)}(E_i, E_j)$ admits natural lifts to $\Z$-grading, parameterized by
vectors $\xi\in N$ such that $\langle \xi,l\rangle=1$ where $l=(0,0,1).$ Indeed, each such element defines an even grading $2\xi$ on the algebra $\C[N\otimes\C^*]$ of functions on the torus,
with the property that $\deg(W)=2.$ Fixing trivializations of all line bundles restricted to the torus, we then obtain the desired grading. It is easy to check that the resulting grading on cohomology  satisfies \eqref{cond}.

Now let us calculate the induced $A_{\infty}$-structure on the algebra $\bigoplus_{\substack{i,j}}\Hom_{D_{sg}(X_0)}(E_i,E_j).$
By Proposition \ref{classification} it suffices to compute the numbers
$$
a=m_n(u_{i-1,i},u_{i-2,i-1},\dots,u_{i,i+1})(0),\quad b=m_n(v_{i+1,i},v_{i+2,i+1},\dots,v_{i,i-1})(0).
$$
We have $a=b$ by symmetry, and by Remark \ref{rm:structures} it is sufficient to show that $a\ne 0.$

\begin{lemma}\label{l:mndsg} In the category $D_{sg}(X_0)$ we have
$a=m_n(u_{i-1,i},u_{i-2,i-1},\dots,u_{i,i+1})(0)\ne 0.$
\end{lemma}
\begin{proof}
Consider the complex of objects in the category $D_{sg}(X_0):$
\begin{equation}\label{complex}
E_1[1-n]\lto E_2[2-n]\lto\dots\lto E_{n-1}[-1],
\end{equation}
where the maps are $u_{i,i+1},$ $1\leq i\leq n-2,$ and we place $E_{n-1}[-1]$ in degree zero.

The convolution of \eqref{complex} is well defined up to an isomorphism. It is isomorphic to $E_n.$
To see this, introduce the divisor
$$
L:= \sum\limits_{k=1}^{\lfloor\frac{n}2\rfloor}\binom {k-1}{2}H_k+\sum\limits_{k=\lfloor\frac{n}2\rfloor+1}^n\left(\binom {n-k}{2}-1\right)H_k.
$$
It is straightforward to check that for $i\geq 0$ the restriction of
$\cO_{X_0}(L-H_1-\dots-H_i)$ to $H_{i+1}$ is trivial.
Moreover, the morphism $u_{i,i+1}:E_i\to E_{i+1}[1]$ for $i\geq 1$ can be interpreted as follows.
Let $f: E_i\cong \cO_{H_i}(L-H_1-\dots-H_{i-1})\to \cO_{\bigcup_{j\ne i}H_j}(L-H_1-\dots-H_i)[1]$
be the morphism corresponding to the extension:
$$
0\to \cO_{\bigcup_{j\ne i}H_j}(L-H_1-\dots-H_i)\to \cO_{X_0}(L-H_1-\dots-H_{i-1})\to
\cO_{H_i}(L-H_1-\dots-H_{i-1})\to 0.
$$
Then $Cone(f)$ is a perfect complex, so $f$ is invertible in $D_{sg}(X_0).$
Let $g$ be the projection
$$
\cO_{\bigcup_{j\ne i}H_j}(L-H_1-\dots-H_i)[1]\lto \cO_{H_{i+1}}(L-H_1-\dots-H_i)[1].
$$
Then $u_{i,i+1}=gf^{-1}.$

By induction, we now see that, for all $1\leq k\leq n-1,$ the following two
properties hold:
\begin{enumerate}
\item the convolution $\cC_k$ of
$\displaystyle
E_1[1-n]\stackrel{u_{1,2}}{\lto} E_2[2-n]\lto\dots
\stackrel{u_{k-1,k}}{\lto} E_{k}[k-n]
$
is isomorphic to $\cO_{H_{k+1}\cup\dots\cup H_n}(L-H_1-\dots-H_k)[k+1-n],$
and
\item the restriction map from $\cO_{H_{k+1}\cup\dots
\cup H_n}(L-H_1-\dots-H_k)[k+1-n]$ (which is isomorphic to $\cC_k$) to $\cO_{H_{k+1}}(L-H_1-\dots-H_k)[k+1-n]
\simeq E_{k+1}[k+1-n]$ corresponds to the morphism
$u_{k,{k+1}}:E_k[k-n]\to
E_{k+1}[k+1-n].$
\end{enumerate}

We conclude that $E_n$ is isomorphic to the
convolution $\cC_{n-1}$ of \eqref{complex}, and that the map from
$\cC_{n-1}$ to $E_n$ induced by $u_{n-1,n}:E_{n-1}[-1]\to E_n$ is
an isomorphism.

Moreover, it is not hard to check that the map from
$E_n$ to \,$\cC_{n-1}$ induced by $u_{n,1}:E_n\to E_1$ is also an isomorphism,
for instance by using an argument similar to the above one to show that
the convolution of
$$E_n[-n]\stackrel{u_{n,1}}{\lto} E_1[1-n]\stackrel{u_{1,2}}{\lto}
E_2[2-n]\lto\dots\stackrel{u_{n-2,n-1}}{\lto} E_{n-1}[-1]$$ is the zero
object.

We claim this implies that $m_n(u_{n-1,n},u_{n-2,n-1},\dots,u_{1,2},
u_{n,1})(0)\neq 0.$ The easiest way to see this
is to use the language of {\it twisted complexes} (see e.g.\ Section 3 of
\cite{seidel-book}). Recall that twisted complexes are a generalization of
complexes in the context of $A_\infty$-categories, for which they provide
a natural triangulated enlargement. The philosophy
is that, in the $A_\infty$ setting, compositions of maps can only
be expected to vanish up to chain homotopies which are explicitly provided
as part of the twisted complex; see Section $3l$ of \cite{seidel-book} for
the actual definition. In our case, the higher compositions
of the morphisms within the complex \eqref{complex} are all zero (since the relevant
morphism spaces are zero), so \eqref{complex} defines a twisted complex
without modification; we again denote this twisted complex by $\cC_{n-1}.$
Moreover, the maps $u_{n,1}$ and $u_{n-1,n}$ induce morphisms of twisted
complexes $\overline{u}_{n,1}\in
\Hom^{\mathrm{Tw}}(E_n,\cC_{n-1})$ and $\overline{u}_{n-1,n}\in
\Hom^{\mathrm{Tw}}(\cC_{n-1},E_n),$ and
by the above argument these are isomorphisms.
Thus the composition
$m_2^{\mathrm{Tw}}(\overline{u}_{n-1,n},\overline{u}_{n,1})$ is an
automorphism of $E_n$; hence the coefficient of $\id_{E_n}$ in this composition is non-zero.
However, by definition of the product in the $A_\infty$-category
of twisted complexes \cite[Equation 3.20]{seidel-book},
$$m_2^{\mathrm{Tw}}(\overline{u}_{n-1,n},\overline{u}_{n,1})=
m_n(u_{n-1,n},u_{n-2,n-1},\dots,u_{1,2},u_{n,1}).$$
It follows that $a\neq 0.$
\end{proof}

The final ingredient needed for Theorem \ref{th:B-side-mn} is the following
generation statement:

\begin{lemma} \label{lem:strict_gen_dbsing}
The objects $E_1,\dots,E_n$ generate the triangulated category $D_{sg}(X_0)$ in the strict sense,
i.e.\ the minimal triangulated subcategory of $D_{sg}(X_0)$ that contains $E_1,\dots,E_n$ coincides
with the whole $D_{sg}(X_0).$
\end{lemma}

\begin{proof}
Clearly, it suffices to show that the sheaves $\cO_{H_1},\dots,\cO_{H_n}$ generate
the category $D^b(\coh(X_0)).$ Denote by $\cT\subset D^b(\coh(X_0))$ the full triangulated subcategory generated by these objects.
As above denote by $\varGamma_{st}$ the intersection $H_s\cap H_t.$

Since the divisors $H_s$ are precisely the irreducible components of $X_0,$ it suffices to prove that
that $D^b_{H_s}(\coh(X_0))\subset \cT$ for all $1\leq s\leq n,$ where $D^b_{H_s}(\coh(X_0))$ is the full subcategory consisting of complexes
with cohomology supported on $H_s.$
We introduce a new ordering on the set of components $H_s$ by setting $s_1=n,$ $s_2=1,$
$s_3=n-1,$ $s_4=2,$ \dots, $s_n=\lfloor\frac{n+1}2\rfloor,$
and will prove by induction on $1\leq i\leq n$ that
\begin{equation}
\label{generated_D^b} D^b_{H_{s_i}}( \coh(X_0))\subset \cT.
\end{equation}

For $i=1$ we have $H_{s_1}=H_n\cong\mathbb{A}^2.$ Therefore, the sheaf $\cO_{H_n}$ generates
$D^b(\coh(H_{s_1}))$ and, hence, it generates $D^b_{H_{s_1}}(\coh (X_0)).$
Thus, the subcategory $D^b_{H_{s_1}}(\coh (X_0))$ is contained in $\cT.$
If $n=3,$ then $H_1\cong H_2\cong\mathbb{A}^2,$ and we are done.

Assume that $n>3,$ and
suppose that \eqref{generated_D^b} is proved for $1\leq i <k.$
By induction hypothesis,
$D^b_{\varGamma_{s_j s_k}}(\coh (X_0))\subset\cT$ for any $j<k.$
The complement $H_{s_k}\setminus (\bigcup_{j<k}\varGamma_{s_j s_k})$ is isomorphic to either $\bbA^2$ (if $k<n-1$)
or an open subset in $\bbA^2$ (if $k=n-1$ or $n$). In any case we obtain that the sheaf $\cO_{H_{s_k}}$
together with the subcategories $D^b_{\varGamma_{s_j s_k}}(\coh (X_0))$ for $j<k$ generate $D^b_{H_{s_k}}(\coh (X_0)).$
In particular, $D^b_{H_{s_k}}(\coh (X_0))\subset\cT.$
This proves \eqref{generated_D^b}
for $i=k,$ which implies that
$\cT=D^b(\coh(X_0)).$
\end{proof}

\section{HMS for cyclic covers}\label{s:covers}

Let $d_1,$ $d_2,$ and $d_3$ be a triple of integers whose sum is a strictly positive integer $D.$  To this data, we shall associate a trivialization of the tangent space of a $D$-fold cyclic cover $C$ of $S^{2} - \{3 \textrm{ points} \},$ as well as a Landau-Ginzburg model on an orbifold quotient of $\C^{3}.$  In order to prove that these are mirror, we shall introduce a purely algebraic model for a category equivalent to a full generating subcategory  of the Fukaya category on one side and of the category of matrix factorizations on the other, then extend Theorem  \ref{th:main} to the cover.

\subsection{A rational grading on $A$}\label{ss:Atilde}
The algebraic model corresponds to a choice of a positive integer $D,$
and of integers $(p_1, p_2, p_3)$ and $(q_1, q_2, q_3)$ such that
$$p_1 + p_2 + p_3 = q_1 + q_2 + q_3 = D \quad \text{and}\quad
  p_i \equiv q_j \equiv D \mod 2.$$
As in  Lemma \ref{l:gradedCW}, we introduce the integers
$d_i = \frac{p_i+q_i}{2}.$
We also introduce the rational numbers
$\tilde{p}_i=p_i/D,$ $\tilde{q}_i=q_i/D,$ and $\tilde{d}_i=d_i/D.$
We then define a \emph{$\frac{1}{D}\Z$-graded} category
$A_{(\tilde{p},\tilde{q})}$ (the notation is analogous to that in Definition \ref{categ}) by setting
\begin{align}
\deg(u_{i-1,i}) & = \tilde{p}_i,\\
 \deg(v_{i,i-1}) & = \tilde{q}_i.
\end{align}
Note that additivity with respect to the multiplicative structure determines the rest of the gradings
\begin{align}
  \deg(x_{i}^{k}) = \deg(y_{i-1}^{k}) & =  2\tilde{d}_ik, \\
  \deg(x_{i}^{k}u_{i-1,i} ) = \deg(u_{i-1,i}  y_{i-1}^{k}) & = \tilde{p}_i+2\tilde{d}_i k,\\
  \deg(y_{i-1}^{k} v_{i,i-1} ) = \deg(v_{i,i-1} x_{i}^{k}) & = \tilde{q}_i+2\tilde{d}_i k.
\end{align}

We will now construct from the
$\frac1D\Z$-graded category $A_{(\tilde{p},\tilde{q})}$ a
$\Z$-graded category $\tilde{A}_{(\tilde{p},\tilde{q})}$,
and discuss $A_\infty$-structures on it.
The process we describe is in fact a specific instance of a
more general construction (see Definition \ref{def:Btilde}).

The first step is to consider an enlargement
$\tilde{A}_{(\tilde{p},\tilde{q})}^{[D]} $  of $A_{(\tilde{p},\tilde{q})}$ in which each object is replaced by $D$ different copies, and the groups of morphisms are shifted by multiples of $\frac{1}{D}.$
(On the symplectic side, the different objects correspond to the components of the inverse image of a curve under a $D$-fold covering map.)
\begin{align}
  \Ob \left( \tilde{A}_{(\tilde{p},\tilde{q})}^{[D]}    \right)  & = \{\tilde{X}_{i}^{k}\,|\, 0 \leq k < D  \} \\
  \tilde{A}_{(\tilde{p},\tilde{q})}^{[D]}(\tilde{X}_{i}^{k}, \tilde{X}_{j}^{\ell}) & =  A_{(\tilde{p},\tilde{q})}
  (   X_{i},  X_{j})\left[\frac{2(\ell - k)}{D}\right].  \label{eq:rational_shift_gradings}
\end{align}
Writing $A_{(1,1)}$ for the $\Z/2$-grading on $A$ in which the generators $u_{i-1,i}$ and $v_{i,i-1}$ both have odd degree, we have a forgetful functor
\begin{equation*}
  \tilde{A}_{(\tilde{p},\tilde{q})}^{[D]} \to A_{(1,1)}
\end{equation*}
 which takes $ \tilde{X}_{i}^{k} $  to $X_{i}.$  This functor is of course not graded, but there is a maximal subcategory of the source with the property that the restriction becomes a $\Z/2$-graded functor:
\begin{defi}\label{def:Atilde}
The category $\tilde{A}_{(\tilde{p},\tilde{q})} $ has objects those of
$\tilde{A}_{(\tilde{p},\tilde{q})}^{[D]} $ and morphisms the subgroup
\begin{equation}
  \tilde{A}_{(\tilde{p},\tilde{q})}(\tilde{X}_{i}^{k},   \tilde{X}_{j}^{\ell} ) \subset A_{(\tilde{p},\tilde{q})}  (   X_{i},  X_{j})\left[\frac{2(\ell - k)}{D}\right]
\end{equation}
generated by morphisms whose degree is integral, and moreover agrees in parity with the degree of the image in $ A_{(1,1)} .$
\end{defi}

We shall also need to understand $A_{\infty}$-structures on
$\tilde{A}_{(\tilde{p},\tilde{q})}.$ For this,
it will be convenient to make the following definition.

\begin{defi} A $\frac1D\Z$-graded $A_{\infty}$-category $B$ consists of a
$\Z/2$-graded $A_{\infty}$-category $B,$ together with
$\frac1D\Z$-gradings on\/ $\Hom^{even}(X,Y)$ and\/ $\Hom^{odd}(X,Y)$ for any pair of objects $X,Y\in Ob(B),$
with respect to which the higher products $m_n$ have degree\/ $2-n.$

A $\frac1D\Z$-graded DG category is a $\frac1D\Z$-graded $A_{\infty}$-category with $m_n=0$ for $n\geq 3,$ and with identity of
degree zero; finally,
a $\frac1D\Z$-graded category is a $\frac1D\Z$-graded DG category with zero differential.\end{defi}

We treat both $A_{(\tilde{p},\tilde{q})}$ and $\tilde{A}_{(\tilde{p},\tilde{q})}^{[D] } $
as $\frac1D\Z$-graded categories, with $u_{i-1,i},$ $v_{i,i-1}$ being odd morphisms. Note that for a $\frac1D\Z$-graded $A_{\infty}$-category $B$ over a field,
the standard construction gives a minimal $A_{\infty}$-structure on the cohomology,
i.e. on the $\frac1D\Z$-graded category $H^{*}(B).$

The $A_\infty$-structures of interest to us arise from the fact that any
$\frac1D\Z$-graded $A_{\infty}$-structure on
$A_{ (\tilde{p},\tilde{q})}$ extends to $ \tilde{A}_{(\tilde{p},\tilde{q})}^{[D] },$
in such a way that $\tilde{A}_{(\tilde{p},\tilde{q})}  $ is an $A_{\infty}$ subcategory.
The following result classifies $\frac1D\Z$-graded $A_{\infty}$-structures on
$A_{ (\tilde{p},\tilde{q})},$  by extending Proposition \ref{classification}:

\begin{prop}\label{classif_rational}
Equation \eqref{a,b} gives a bijection between the set of $\frac1D\Z$-graded
$A_{\infty}$-structures on $A_{ (\tilde{p},\tilde{q})},$ up to
$\frac1D\Z$-graded strict homotopy,  and  $\kk^{2}.$
\end{prop}

\begin{proof}The proof is the same as for Proposition \ref{classification} (2). Namely, Hochschild cohomology can be
defined for $\frac1D\Z$-graded categories in exactly the same manner as in the $\Z$-graded case, and all the relevant
computations from Sections \ref{s:Ainfinity} and \ref{s:classif} still hold in this setting.\end{proof}

\begin{cor}
The  $A_{\infty}$ structure on $\tilde{A}_{(\tilde{p},\tilde{q})}  $
induced by a $\frac1D\Z$-graded $A_{\infty}$-structure on
$A_{(\tilde{p},\tilde{q})}$  depends, up to strict $\Z$-graded homotopy,
only on the constants $a$ and $b$ appearing in Equation \eqref{a,b}.
\end{cor}
\proof
A strict homotopy between two $A_{\infty}$ structures on
$A_{ (\tilde{p},\tilde{q})}$ extends to one between the structures on
$\tilde{A}_{(\tilde{p},\tilde{q})}^{[D]}.$  Moreover, if the homotopy is
graded, the functor will preserve integral gradings, and hence induce a
functor on the integral subcategories.
\endproof

The next result will allow us some flexibility in proving homological mirror symmetry by choosing an appropriate graded representative of each object.  The key observation needed for its proof is that if we allow arbitrary integers in Equation \eqref{eq:rational_shift_gradings}, then replacing $k$ by $k+D$ corresponds to a homological shift by $2,$ so that integrality is preserved as well as parity:
\begin{lemma} \label{lem:closure_shift_iso}
The closure of $\tilde{A}_{(\tilde{p},\tilde{q})}  $ under the shift functor depends, up to isomorphism, only the triple $(d_1,d_2,d_3).$
\end{lemma}
\proof
Let $(p_1', p_2', p_3')$ and $(q_1', q_2', q_3')$ be triples of integers such that
\begin{equation*}
  p_i' + q_i' = p_i + q_i.
\end{equation*}
The assignment
\begin{align*}
 \tilde{X}_{1}^{k} &  \mapsto \tilde{X}_{1}^{k}\\
\tilde{X}_{2}^{k} &  \mapsto \tilde{X}_{2}^{{k+ p_2 - p_2'}} \\
\tilde{X}_{3}^{k} &  \mapsto \tilde{X}_{3}^{{k+ p_2 - p_2'+p_3- p_3'}}
\end{align*}
defines a  $\frac1D\Z$-graded isomorphism, and hence an isomorphism of the corresponding subcategories of integrally graded morphisms.
\endproof

\subsection{The wrapped Fukaya category of a cyclic cover}
As in the previous section, we choose integers $(d_1,d_2,d_3)$ whose sum is a strictly positive integer $D.$  Projecting the Riemann surface
\begin{equation}
  C = \{ (x,y) | y^{D} =  x^{d_2}(1-x)^{d_3} \} \subset \C \times \C^{*}
\end{equation}
to the $x$-plane defines a cover of $\C - \{0,1\},$ in which the punctures are ordered $(\infty,0,1).$

\begin{prop} \label{prop:wrapped_cover}
The wrapped Fukaya category of  $C,$ with the $\Z$-grading determined by
the restriction of the holomorphic $1$-form
$\frac{dx}{y},$  is strictly generated by the components of  the inverse image of the real axis.  Whenever $p_i + q_i = 2d_i,$ there is a choice of grading on these components so that the resulting subcategory of the Fukaya category is $A_{\infty}$-equivalent to the structure induced by
$m^{1,1}$ on $  \tilde{A}_{(\tilde{p},\tilde{q})} .$
\end{prop}

\begin{remark}
{\rm
A description of the Fukaya categories of covers as a semi-direct product
has previously appeared in the proof of Homological mirror symmetry for
the closed genus $2$ curve (see \cite[Remark 8.1]{Segenus2}), and in
Sheridan's work \cite[Section 7]{sheridan}, but our implementation will
be quite different because we are concerned with recovering integral gradings that do not come from trivializations of the tangent space of $C$ which are pulled back from the base.  Of course, underlying either approach is the fact that each holomorphic disc in the base lifts uniquely, upon choosing a basepoint, to a holomorphic disc in the cover.
}
\end{remark}


In order to prove Proposition \ref{prop:wrapped_cover}, we choose our curves to be
\begin{align*}
L_1 & = (-\infty,0)  \\
L_2 & = (0,1) \\
L_{3}& = (1,+\infty).
\end{align*}
Note that each component of the inverse image of $L_i$ in $C$ has constant phase with respect to the $1$-form $\frac{dx}{y}.$ The different components are distinguished by their phases: those lying over $L_2$ have phases the $D$-th roots of unity, while the inverse images of $L_1$ and $L_3$ respectively have phases equal to the solutions of $y^{D}=(-1)^{d_2}$ and $y^{D}=(-1)^{d_3}.$  If we fix the exponential map
\begin{equation*}
  \alpha \mapsto e^{\pi \sqrt{-1} \alpha}
\end{equation*}
then the graded lifts of such components are again distinguished by the corresponding real-valued phase, which lies in $\frac{d_i}{D} + \frac{2}{D}\Z.$  For each integer $0 \leq k < D,$ we fix graded lifts $\tilde{L}_{i}^{k}$ of $L_i$ with real valued phases
\begin{equation*}
  \textrm{Phase}(\tilde{L}_{i}^{k}) = \begin{cases}
\frac{-d_2}{D} + \frac{2k}{D} & \textrm{ if } i=1 \\
  \frac{2k}{D}  & \textrm{ if } i=2 \\
\frac{d_3}{D} + \frac{2k}{D}  & \textrm{ if } i=3.
\end{cases}
\end{equation*}

If we use a Hamiltonian on $C$ which is pulled back from $\C - \{0,1\},$ a chord between $\tilde{L}_{i}^{k}$ and $\tilde{L}_{j}^{\ell}$ is uniquely determined by its projection to $\C,$ which is a chord with endpoints on $L_i$ and $L_j.$  Choosing the Hamiltonian as in Section \ref{s:wrapped}, the differential in the Floer complex vanishes, so that $HW^{*}( \tilde{L}_{i}^{k}, \tilde{L}_{j}^{\ell})$ is the subgroup of $ HW^{*}( L_i,L_j )$ generated by those chords admitting a lift with the correct boundary conditions.

By construction, we have arranged for the chords $v_{2,1}$ and $u_{2,3}$ to respectively lift to generators of $HW^{*}( \tilde{L}_{2}^{0}, \tilde{L}_{1}^{0}) $ and  $HW^{*}( \tilde{L}_{2}^{0}, \tilde{L}_{3}^{0}) .$  It is then not hard to see that the generators of $HW^{*}( \tilde{L}_{2}^{0}, \tilde{L}_{1}^{0}) $ correspond to lifts of chords $v_{2,1}x_{2}^{k}$ whenever $D$ divides $kd_2,$ while the generators of
$HW^{*}( \tilde{L}_{2}, \tilde{L}_{3}) $ are lifts of $ y_{2}^{k} u_{2,3}$ where $D$ divides $kd_3.$

Note that if we set $q_2=p_3=D,$ these are precisely the monomials in
$A_{(\tilde{p},\tilde{q})}(X_2,X_1) $ and $A_{(\tilde{p},\tilde{q})}(X_2,X_3)$
of odd integer degree, i.e.\ the generators of $\tilde{A}_{(\tilde{p},\tilde{q})}(\tilde{X}_2^{0},\tilde{X}_1^{0})$
and $\tilde{A}_{(\tilde{p},\tilde{q})}(\tilde{X}_2^{0},\tilde{X}_3^{0}).$  Extending this computation from $k=\ell=0$ to the general case, and using the fact that a holomorphic curve in $\C - \{0,1\}$ lifts uniquely to $C$ upon choosing a basepoint, we conclude:
\begin{lemma}
If $(p_1,p_2,p_3) = (D - 2d_2 , 2 d_2 -  D, D)$ and $(q_1, q_2, q_3) = (D - 2d_3  , D, 2d_3- D ),$ then the subcategory of $\W(C)$ with objects $\tilde{L}_{i}^{k}$ is quasi-isomorphic to $\tilde{A}_{(\tilde{p},\tilde{q})}$ equipped with the $A_{\infty}$ structure induced by $m^{1,1}.$  \qed
\end{lemma}

This result, together with Lemma \ref{lem:closure_shift_iso}, implies the second part of Proposition \ref{prop:wrapped_cover}, while the first part follows from Theorem \ref{thm:thimbles_generate} applied to the composition of the covering map from $C$ to $\C - \{0,1\}$ with the Lefschetz fibration used in Lemma \ref{l:generate}.

\subsection{Equivariant Landau-Ginzburg mirror model}
\label{ss:LGquotient}

Consider $\C^3$ equipped with the diagonal action of $G=\Z/D$ with weights $\frac1D(d_1,d_2,d_3),$ where $d_i=\frac{p_i+q_i}2$ as above. Let
$W:=z_1z_2z_3\in\C[z_1,z_2,z_3]^G.$ Our LG model is $(\C^3//G,W).$ We have an equivalence
\begin{equation}D_{sg}^G(W^{-1}(0))\cong MF^G(W).\end{equation}

For each $\chi\in G^*\cong\Z/D,$ we have a functor $-(\chi)$ on $D_{sg}(W^{-1}(0)).$ For each $0\leq k<D,$ denote by $\chi_k\in G^*$
the character corresponding to the image of $k$ in $\Z/D.$ Take the objects
$$E_i^k:=\cO_{H_i}(\chi_k)\in D_{sg}^G(W^{-1}(0)),\quad 1\leq i\leq 3,\quad 0\leq k<D,$$
where $H_i=\{z_i=0\}\subset W^{-1}(0).$ Clearly, they generate (strictly) the category $D_{sg}^G(W^{-1}(0)).$
Now we would like to prove that there is an equivalence $D\W(C)\cong D_{sg}^G(W^{-1}(0)),$ such that the objects $\tilde{L}_i^k$ correspond to $E_i^k.$
To do that, we will deal with $\frac1D\Z$-gradings on matrix factorizations.

Put $\deg(z_i):=2\tilde{d}_i=\frac{2d_i}D.$ Then the algebra $R=\C[z_1,z_2,z_3]$ becomes $\frac1D\Z$-graded, and $\deg(W)=2.$
Define a $\frac1D\Z$-graded DG category $MF^{\frac1D\Z}(W)$ of $\frac1D\Z$-graded matrix factorizations as follows.

An object of this category is a pair of free finitely generated $\frac1D\Z$-graded $R$-modules $\underline{T}=(T_1, T_0),$ together with homogeneous morphisms $t_1: T_1\to T_0,$ $t_0: T_0\to T_1$ of degree $1,$
such that $t_1 t_0=W\cdot\id_{T_0},$ $t_0 t_1=W\cdot\id_{T_1}.$

Further, for two objects $\underline{T}, \underline{S},$ the $2$-periodic complex of morphisms $\Hom(\underline{T},\underline{S})$ is defined as usual.
Composition is also the usual one. Finally, the $\frac1D\Z$-grading on $\Hom^{even}(\underline{T}, \underline{S})$ and $\Hom^{odd}(\underline{T},
\underline{S})$ comes from
the $\frac1D\Z$-gradings on $T_1,$ $T_0,$ $S_1,$ $S_0.$

It is straightforward to check that we get indeed a $\frac1D\Z$-graded DG category. Now we consider three particular matrix factorizations
${\underline{T}\,}_1, {\underline{T}\,}_2, {\underline{T}\,}_3\in MF^{\frac1D\Z}(W)$ as follows:
$$
{\underline{T}\,}_1=\{R\stackrel{z_2z_3}{\lto}R[1-2\tilde{d}_1]\stackrel{z_1}{\lto}R\},
$$
and analogously for ${\underline{T}\,}_2, {\underline{T}\,}_3.$ Denote by $\cC_{d_1,d_2,d_3}\subset MF^{\frac1D\Z}(W)$ the full $\frac1D\Z$-graded DG subcategory
with objects ${\underline{T}\,}_1, {\underline{T}\,}_2, {\underline{T}\,}_3.$ Then the $\frac1D\Z$-graded cohomological category $H^{*}(\cC_{d_1,d_2,d_3})$ is equipped with a natural minimal $A_{\infty}$-structure (defined up to graded strict homotopy).

For convenience, set ${\underline{T}\,}_{i+3}:={\underline{T}\,}_i,$ $z_{i+3}:=z_i,$ and $d_{i+3}:=d_i.$

\begin{prop}$(1)$ There is a natural equivalence of $\frac1D\Z$-graded categories $A_{(\tilde{p},\tilde{q})}\cong H^{*}(\cC_{d_1,d_2,d_3}),$
where $p_i=2d_i+2d_{i+1}-D$ and $q_i=2d_{i-2}+2d_{i-1}-D.$

$(2)$ Under the above equivalence, the $A_{\infty}$-structure on $H^{*}(\cC_{d_1,d_2,d_3})$ is homotopic to $m^{1,1}.$\end{prop}

\begin{proof}(1) For each $i=1,2,3,$ consider the odd closed morphism $\tilde{u}_{i-1,i}: {\underline{T}\,}_{i-1}\to {\underline{T}\,}_i$ given by the pair of morphisms
$$
R\stackrel{z_{i+1}}{\lto} R[1-2\tilde{d}_i],\quad R[1-2\tilde{d}_{i-1}]\stackrel{-1}{\lto} R.
$$
The sign appears because the morphism is odd.
Clearly, $\deg(\tilde{u}_{i-1,i})=\frac{p_i}D=\tilde{p}_i.$ Similarly, consider the odd morphism
$\tilde{v}_{i,i-1}: {\underline{T}\,}_i\to {\underline{T}\,}_{i-1}$ given by the pair of morphisms:
$$
R\stackrel{z_{i-2}}{\lto} R[1-2\tilde{d}_{i-1}],\quad R[1-2\tilde{d}_i]\stackrel{-1}{\lto} R.
$$
It is easy to see that $\deg(\tilde{v}_{i,i-1})=\tilde{q}_i.$ Moreover, the compositions $\tilde{u}_{i+1,i}\tilde{u}_{i-1,i}$
and $\tilde{v}_{i,i-1}\tilde{v}_{i+1,i}$ are homotopic to zero. Hence, we have a functor
$$
A_{(\tilde{p},\tilde{q})}\lto H^{*}(\cC_{d_1,d_2,d_3})
$$
of $\frac1D\Z$-graded categories. It is easily checked to be an equivalence.

(2) The non-vanishing of the constant terms of $m_3(\tilde{u}_{3,1},\tilde{u}_{2,3},\tilde{u}_{1,2})$ and $m_3(\tilde{v}_{2,1},\tilde{v}_{3,2},\tilde{v}_{1,3})$
follows from the results of Section \ref{s:Dbsing}. Indeed these constants terms do not depend on gradings, and they were shown not to vanish for integer gradings. Hence, the statement follows from Proposition \ref{classif_rational}.\end{proof}

\begin{defi}\label{def:Btilde}
For a $\frac1D\Z$-graded $A_{\infty}$-category $B,$ denote by $\tilde{B}$ the $\Z$-graded $A_{\infty}$-category
whose objects are pairs $(X,k),$ where $X\in Ob(B)$ and $0\leq k<D,$ and where morphisms are defined by the formula
\begin{align*} \label{eq:morphisms_cover_category} \Hom_{\tilde{B}}^{2i}((X,k),(Y,l)) & = \Hom^{2i+\frac{2(l-k)}D,even}(X,Y) \\
 \Hom_{\tilde{B}}^{2i-1}((X,k),(Y,l)) &= \Hom^{2i-1+\frac{2(l-k)}D,odd}(X,Y).\end{align*}
The higher products are induced by those of $B.$\end{defi}

\noindent (Compare with the construction in Section \ref{ss:Atilde}.)

It is clear that the assignment $B\mapsto\tilde{B}$ defines a functor from $\frac1D\Z$-graded $A_{\infty}$-categories and $A_{\infty}$-morphisms
to usual $\Z$-graded $A_{\infty}$-categories and $A_{\infty}$-morphisms.

\begin{cor}With the same notation, the DG category $\widetilde{\cC_{d_1,d_2,d_3}}$ is quasi-equivalent to the $A_{\infty}$-category
$(\tilde{A}_{(\tilde{p},\tilde{q})},\tilde{m}^{1,1}),$ where the $A_{\infty}$-structure $\tilde{m}^{1,1}$ is induced by $m^{1,1}.$\end{cor}

Now write the matrix factorizations in $MF^G(W)$ corresponding to the above generators $E_i^k\in D_{sg}^G(W^{-1}(0)):$
$$
{\underline{\tilde{T}}\,}_i^k=\{R(\chi_k)\stackrel{z_{i+1}z_{i+2}}{\longrightarrow} R(\chi_{k-d_i})\stackrel{z_i}{\lto} R(\chi_k)\}.
$$

Then it is straightforward to see that we have a fully faithful functor of $\Z/2$-graded DG categories
$$\widetilde{\cC_{d_1,d_2,d_3}}\lto MF^G(W),\quad ({\underline{T}\,}_i, k)\mapsto {\underline{\tilde{T}}\,}_i^k.$$

Since the collection of sheaves $\{ \cO_{H_i}(\chi_k) \}_{k=0}^{D-1}$ strongly generate the category of equivariant coherent sheaves on $W^{-1}(0)$ supported on the component $H_i$, we obtain the following result using the same argument as the proof of Lemma \ref{lem:strict_gen_dbsing}:

\begin{prop}\label{D_sg_descr}The triangulated category
$D_{sg}^G(W^{-1}(0))$ is strictly generated by the objects $E_i^k$ introduced above. The resulting $\Z/2$-graded
DG subcategory of $MF^G(W)$ is quasi-equivalent to the $(\Z/2$-graded$)$ $A_{\infty}$-category $\tilde{A}_{(\tilde{p},\tilde{q})}.$\end{prop}

Taking into account the results of the previous Subsection, we have proved the following theorem.

\begin{theo}The triangulated categories $D\W(C)$ and $D_{sg}^G(W^{-1}(0))$ are equivalent.
\end{theo}

\begin{proof}This follows from Proposition \ref{D_sg_descr}, Proposition \ref{prop:wrapped_cover} and Lemma \ref{lem:closure_shift_iso}.
\end{proof}

\appendix

\section{A generation result for the wrapped Fukaya category}
\label{appendix}

Throughout this section, we shall consider $\pi \co \Sigma \to D^2,$ a Lefschetz fibration on a compact Riemann surface with boundary,
i.e. a simple branched covering of the disc.  The inverse image of an arc starting at a critical value and ending at $1 \in D^2$
is called a Lefschetz thimble, and the collection  of thimbles obtained by choosing a collection of arcs which do not intersect
in the interior, one for each critical point, is called a \emph{basis of thimbles}.

\begin{theo} \label{thm:thimbles_generate}
Any basis of thimbles generates (in the strict sense) the wrapped Fukaya category of\/ $\Sigma$ for all coefficient rings.
\end{theo}

Note that this result is stronger than the split-generation statement that might be expected by applying the results of \cite{generate}.
 We shall prove it by embedding $\Sigma$ inside a larger Riemann surface where the Lagrangians we consider extend to circles.
 Then, following the strategy developed by Seidel in \cite{seidel-book}, we apply the long exact sequence for a Dehn twist
 to derive a generation statement in the Fukaya category of \emph{compact} Lagrangians.  Finally, we shall use the existence
 of a restriction functor constructed in \cite{AS} to conclude the desired result. We shall omit discussions of signs and gradings
 (and the corresponding geometric choices) which essentially play no role in our arguments.

Let us therefore start by choosing a Liouville structure on $\Sigma,$ i.e. a $1$-form $\lambda$ whose differential is symplectic,
and whose associated Liouville flow is outward pointing at the boundary.

In addition to mere exactness, the construction of a restriction functor will require us to consider the following technical condition on a curve $\alpha \in \Sigma$
\begin{equation}
\label{eq:strong_exactness}
\parbox{30em}{$\lambda|\alpha$ has a primitive function which vanishes on the boundary.}
\end{equation}
Choosing a basis of thimbles, we replace $\lambda$ (adding the differential of a function) so that this condition holds for each element of the basis.  For more general curves, we have:
\begin{lemma}
Every exact curve in $\Sigma$ is equivalent, in the wrapped Fukaya category, to a curve satisfying Condition \eqref{eq:strong_exactness}.
\end{lemma}
\begin{proof}
The quasi-isomorphism class of a curve is invariant under Hamiltonian isotopies in the completion of $\Sigma$ to a surface of infinite area.
We leave the (easier) non-separating case to the reader, and assume we are given a curve $\alpha_0$ whose union with a subset of $\partial \Sigma$
(consisting of an interval together with some components) bounds a submanifold  $\Sigma_0.$  Stokes's theorem implies that difference between
the values of a primitive at the two endpoints of $\alpha_0$ equals
\begin{equation*}
\int_{\Sigma_0} \omega - \int_{\Sigma_0\cap \partial \Sigma} \lambda
\end{equation*}
where each component of $ \Sigma_0 \cap \partial \Sigma $ is given the orientation induced as a subset of the boundary of $\Sigma.$
Note that the integral over the boundary is strictly greater than $0$ and smaller than the area of $\Sigma.$
In particular, we may isotope $\alpha_0,$ through embedded curves which have the same boundary, to a curve $\alpha_1$ bounding a surface $\Sigma_1$
of area exactly $ \int_{\Sigma_0\cap \partial \Sigma} \lambda .$  Stokes's theorem now implies that any primitive on $\alpha_1$
must have equal values at the endpoints.   The isotopy between $\alpha_0$ and $\alpha_1$ can be made Hamiltonian after enlarging $\Sigma$
by attaching infinite cylinders to its boundary components.
\end{proof}

To prove that thimbles generate the wrapped Fukaya category, it suffices therefore to prove that an arbitrary curve $\gamma,$
satisfying Condition \eqref{eq:strong_exactness}, is equivalent to an iterated cone built from thimbles.
We consider the Riemann surface $\Sigma_{\gamma}$ obtained by attaching a $1$-handle along the boundary of $\gamma.$
Weinstein's theory of handle attachment gives a Liouville form on $\Sigma_{\gamma}$ for which the inclusion of $\Sigma$ is a
subdomain, and such that the union of $\gamma$ with the core of the new handle is an exact Lagrangian circle which we shall denote $\gamma_0.$
In addition, we may construct a Lefschetz fibration
\begin{equation*}
  \pi_{\gamma} \co \Sigma_{\gamma} \to D^{2}(1+\epsilon)
\end{equation*}
over the disc of radius $1+\epsilon,$ whose restriction to $\Sigma$ agrees with $\pi,$ and which has exactly one critical point outside the unit disc.

Let us choose a basis of thimbles for $\pi_{\gamma}$ extending the previous basis, and such that the additional arc does not enter the unit disc.
We then consider a double cover of $\Sigma_{\gamma}$ denoted $ \tilde{\Sigma}_{\gamma} ,$ which is branched at the inverse image of $1+\epsilon.$
The thimbles of $\pi_{\gamma}$ double to exact Lagrangian circles $(\gamma_1, \ldots, \gamma_{d}, \gamma_{d+1})$ in $ \tilde{\Sigma}_{\gamma},$
with the convention that $\gamma_{d+1}$ is the double of the thimble coming from the new critical point.  Since $\gamma_0$ does not link the branching point,
its inverse image in $ \tilde{\Sigma}_{\gamma}  $ consists of a pair of curves which we shall denote $\gamma_{\pm}.$

The following result is essentially Proposition 18.15 of \cite{seidel-book}. Its proof relies on the correspondence between
algebraic and geometric Dehn twists, and the fact that applying a series of Dehn twists about the curves $\gamma_1, \ldots, \gamma_{d+1}$
maps $\gamma_+$ to a curve isotopic to $\gamma_-.$
\begin{lemma} \label{lem:twist_trick}
The direct sum  of $\gamma_+$  with an object geometrically supported on $\gamma_{-}$ lies in the category generated by
$(\gamma_1, \ldots, \gamma_{d}, \gamma_{d+1}).$ \qed
\end{lemma}

Proposition 18.15 of \cite{seidel-book} in fact describes the precise object supported on $\gamma_-$ which
appears in this Lemma; as this is inconsequential for our intended use we avail ourselves of the option of omitting any discussion of signs and gradings.
We complete this appendix with the proof of its main result:

\begin{proof}[Proof of Theorem \ref{thm:thimbles_generate}]
The inverse image of $\Sigma$ in $\tilde{\Sigma}_{\gamma} $ consists of two components; by fixing the one including $\gamma_+,$ we obtain an inclusion
\begin{equation*}
  \iota \co \Sigma \to \tilde{\Sigma}_{\gamma},
\end{equation*}
which is again an inclusion of Liouville subdomains for an appropriate choice of Liouville form on the double branched cover.

By construction, $\gamma_{-}$ and $\gamma_{d+1}$ are disjoint from $\iota(\Sigma),$ while $\gamma_+$ intersects $\iota(\Sigma)$ in
$\gamma$ and $(\gamma_{1}, \ldots, \gamma_{d})$ in the originally chosen basis of vanishing cycles.  Since, by construction,
Condition \eqref{eq:strong_exactness} holds for these curves, we may apply the restriction functor defined in Sections 5.1 and 5.2 of
\cite{AS}.  This $A_{\infty}$ functor, defined on the subcategory of the Fukaya category of $ \tilde{\Sigma}_{\gamma} $
consisting of objects supported on one of the curves $(\gamma_+, \gamma_-, \gamma_1,  \ldots, \gamma_{d}),$
has target the wrapped Fukaya category of $ \iota(\Sigma) $ and takes a curve to its intersection with the subdomain.
By Lemma \ref{lem:twist_trick}, the direct sum of $\gamma_+$ and an object supported on $\gamma_-$ lies in the category
by generated by $ (\gamma_1, \ldots, \gamma_{d}, \gamma_{d+1}) .$  Since  $\gamma_-$ is disjoint from $\iota(\Sigma),$
the image of this direct sum under restriction is $\gamma,$ so we conclude, as desired, that $\gamma$ lies in the category generated by thimbles.
\end{proof}

\section{Idempotent completion}

The purpose of this appendix is to prove that the triangulated category of singularities $D_{sg}(X_0)$ of the singular fiber
$X_0=W^{-1}(0)$ of the LG model $(X(n), W)$ is idempotent complete. This implies that the
derived wrapped Fukaya category $D\W(C)$ is also idempotent complete.

A full triangulated subcategory $\cN$ of a triangulated category $\cT$
is called {\em dense} in $\cT$ if each object of $\cT$ is a direct
summand of an object isomorphic to an object in $\cN.$
An amazing theorem of R.~Thomason \cite[Th. 2.1]{Th} asserts that
there
is an one-to-one correspondence between strictly full dense
triangulated subcategories $\cN$ in $\cT$ and subgroups $H$ of the
Grothendieck group $K_0(\cT).$
Moreover, we know that under this correspondence $\cN$ goes to the image of $K_0(\cN)$ in
$K_0(\cT)$ and to $H$ we attach the full subcategory $\cN_H$ whose
objects are those $N$ in $\cT$ such that $[N]\in H \subset K_0(\cT).$
Actually, in this situation map from $K_0(\cN)$ to $K_0(\cT)$ is an inclusion.

Let us consider the triangulated category of singularities $D_{sg}(Z)$ for some scheme $Z.$
The Grothendieck group $K_0(D_{sg}(Z))$ is equal to the cokernel of the map $K_0(\perf{Z})\to K_0(\db{\coh Z}).$

On the other hand, by
\cite[Th. 9]{Schl} (see also \cite[Th. 5.1]{Ke}) there is a long exact sequence for K-groups
$$
\cdots\lto K_i(\perf{Z})\lto K_i(\db{\coh Z})\lto K_i(\overline{D_{sg}(Z)})\lto
K_{i-1}(\perf{Z})\lto \cdots
$$
where $\overline{D_{sg}(Z)}$ is the idempotent closure (or {\it Karoubian completion}) of $D_{sg}(Z).$

Using the fact that $K_{-1}$  is trivial for
a small abelian category (\cite[Th.~6]{Schl}),
we obtain a short exact sequence
$$
0\lto K_0(D_{sg}(Z))\lto K_0(\overline{D_{sg}(Z)})\lto
K_{-1}(\perf{Z})\lto 0.
$$
This sequence shows that $K_{-1}(\perf{Z})$ is a measure of the difference
between $D_{sg}(Z)$ and its idempotent
completion $\overline{D_{sg}(Z)}.$

To summarize all these results, we obtain the following proposition:

\begin{prop} The triangulated category of singularities $D_{sg}(Z)$ is idempotent complete if and only if
$K_{-1}(\perf{Z})=0.$
\end{prop}

Also recall that the negative $K$-groups are defined
by induction from the following exact sequences
\begin{multline*} 0\to K_i(\perf{Z})\to K_i(\perf{Z[t]})\oplus
K_i(\perf{Z[t^{-1}]})\to K_i(\perf{Z[t,t^{-1}]})\to\\
\to K_{i-1}(\perf{Z})\to 0.
\end{multline*}

In particular, the group  $K_{-1}(\perf{Z})$ is isomorphic to the
cokernel of the canonical map $K_0(\perf{Z[t]})\oplus
K_0(\perf{Z[t^{-1}]})\to K_0(\perf{Z[t,t^{-1}]}).$

Now we consider the specific case $Z=X_0,$ where $X_0$ is the singular fiber of\/ $W: X(n)\to \bbC$ defined in Section \ref{s:LGmodel},
i.e.\ the union of the toric divisors of $X(n).$

\begin{prop} Let $X_0$ be as above, then $K_{-1}(\perf{X_0})=0.$
\end{prop}
\begin{proof}
Let us denote by $\varGamma\subset X_0$ the one-dimensional subscheme consisting of the singularities of $X_0,$ i.e.\ the union
of all the toric curves in $X(n).$ Denote by $\pi:X_0'\to X_0$ the normalization of $X_0$ and set $\varGamma'=\varGamma\times_{X_0} X_0'.$
 By \cite[Th. 3.1]{We} there is a long exact sequence of $K$-groups which in this case gives
 the following exact sequence:
 $$
 K_0(X_0')\oplus K_0(\varGamma)\lto K_{0}(\varGamma')\lto K_{-1}(X_0)\lto K_{-1}(X_0')\oplus K_{-1}(\varGamma),
 $$
 where all $K$\!-groups are $K$\!-groups of perfect complexes.
Since the normalization $X_0'$
is the disjoint union of smooth toric surfaces we have $K_{-1}(X_0')=0.$ Considering components of the normalization $X_0'$
it is also easy to deduce that the restriction
map $K_0(X_0')\to K_0(\varGamma')$ is surjective. Thus it is sufficient to show that $K_{-1}(\varGamma)$ is trivial.

To any Noetherian curve $C$ we can associate a bipartite graph $\gamma$ defined as follows.
The graph $\gamma$ has one vertex for each singular point $s$ of $C$ and one vertex for each component of the normalization $p: C'\to C.$
For each point of $p^{-1}(s)$ there is an edge connecting the corresponding component of $C'$ with the singular point $s$ of $C.$

By \cite[Lemma 2.3]{We} there is an isomorphism $K_{-1}(C)=\bbZ^{\lambda},$ where $\lambda$ is the number of loops in the bipartite graph $\gamma$
associated to $C.$
It is easy to see that in our case the bipartite graph of $\varGamma$ does not have any loop. Thus $K_{-1}(\varGamma)=0,$ and
$K_{-1}(X_0)=0$ too.
\end{proof}

\end{document}